\documentclass[11pt]{amsart}

\usepackage{tikz}
\usetikzlibrary{arrows.meta,positioning,calc}

\tikzset{
  geoschematic/.style={
    >=Latex,
    line cap=round,
    line join=round,
    line width=0.6pt,
    font=\small,
  },
  geolabel/.style={font=\small},
}

\usepackage{graphicx}

\usepackage{amsmath,amssymb,amsfonts,amsthm}
\usepackage{mathrsfs}

\DeclareMathOperator{\supp}{supp}

\usepackage[numbers,sort&compress]{natbib}
\usepackage[colorlinks=true,linkcolor=blue,citecolor=blue,urlcolor=blue]{hyperref}

\theoremstyle{plain}
\newtheorem{theorem}{Theorem}[section]
\newtheorem{lemma}[theorem]{Lemma}
\newtheorem{proposition}[theorem]{Proposition}
\newtheorem{corollary}[theorem]{Corollary}

\theoremstyle{definition}
\newtheorem{definition}[theorem]{Definition}
\newtheorem{assumption}[theorem]{Assumption}

\theoremstyle{remark}
\newtheorem{remark}[theorem]{Remark}

\title{Heat Flow under Semi-Flat Collapse with Conic Renormalization}
\author{Xin yu Liao}
\date{}
\begin{document}

\maketitle
\begin{abstract}
Motivated by the Strominger--Yau--Zaslow (SYZ) collapse of elliptic K3 surfaces, we study heat flow
on Ricci-flat K\"ahler manifolds $(X_t,g(t))$ equipped over $B_{\mathrm{reg}}=B\setminus D$ with a semi-flat
$2$-torus fibration $\Pi_t:X_t\to B$. Assuming an exponentially accurate semi-flat product approximation
and a uniform vertical spectral gap at the collapse scale, we prove that for each fixed $\tau>0$ the
fiber-compressed heat operators $P_t e^{-\tau\Delta_{g(t)}} I_t$ converge strongly on $L^2$ to the base
heat semigroup $e^{-\tau\Delta_{g_B}}$; equivalently, bilinear pairings of the total-space heat kernel
against fiber-constant test functions converge on compact subsets of $B_{\mathrm{reg}}$.

Near the discriminant $D$ the limiting base metric is conic. We introduce a canonical renormalized
bilinear functional $K_B^{\mathrm{ren}}$ by replacing the pairing in neighborhoods of each conic point
with the model flat-cone heat kernel. A conic parametrix proves cutoff- and chart-independence of
$K_B^{\mathrm{ren}}$ and shows that the total-space pairings converge to $K_B^{\mathrm{ren}}(\Phi,\Psi;\tau)$
for all $\Phi,\Psi\in C_c^{\infty}(B)$. Under exponentially small geometric errors we obtain an explicit
exponential convergence rate, uniform for $\tau$ in compact subsets of $(0,\infty)$, and we also treat
Neumann and nonnegative Robin boundary conditions.
\end{abstract}

\section{Introduction and statement of results}

\subsection{Motivation: heat flow under SYZ-type collapse}
Collapsing families of Ricci--flat K\"ahler metrics arise naturally in complex geometry and mirror symmetry,
most prominently in the Strominger--Yau--Zaslow (SYZ) picture \cite{SYZ} and in the degeneration theory of
elliptically fibered K3 surfaces.  In such settings one expects that, away from singular fibers,
the geometry becomes increasingly close to a semi-flat product: a fixed base metric coupled with flat torus
fibers whose diameters shrink to zero, with ``instanton corrections'' that are exponentially small in the
collapse scale; see, for instance, \cite{GrossWilson,TosattiAdiabatic,GrossTosattiZhang}.Analytically, this collapse suggests a clean separation between low-energy horizontal
dynamics and high-energy vertical oscillations, governed by a uniform fiberwise spectral gap.
Such spectral separation is consistent with the adiabatic limit framework developed in 
\cite{TosattiAdiabatic,GrossTosattiZhang}.

The purpose of this paper is to make this picture precise at the level of heat flow.
Since pointwise convergence of heat kernels is generally unstable under collapse, we focus on a robust notion:
\emph{bilinear heat-kernel pairings} tested against fiber-constant lifts of base functions, or equivalently the strong
convergence of \emph{fiberwise-compressed} heat semigroups.
The new feature occurs at the discriminant: the limiting base geometry is conic near $D$, and to treat test functions
that meet $D$ one needs an intrinsic way to pass from purely interior convergence on $B_{\mathrm{reg}}$ to a global,
cutoff-independent limiting bilinear object. We introduce such a limit functional by a canonical conic renormalization
based on the flat-cone model.

\subsection{Geometric setting and analytic identifications}
Let \(B\) be a smooth surface, let \(D=\{p_1,\dots,p_N\}\subset B\) be a finite discriminant set, and set
\[
B_{\mathrm{reg}}:=B\setminus D.
\]
Let \((X_t,g(t))\) be a family of Riemannian manifolds equipped with smooth surjective submersions
\(\Pi_t:X_t\to B\) whose fibers \(F_{b,t}:=\Pi_t^{-1}(b)\) are flat $2$-tori over \(B_{\mathrm{reg}}\).
Fix a precompact domain \(B_0\Subset B_{\mathrm{reg}}\) with smooth boundary and set \(X_{t,0}:=\Pi_t^{-1}(B_0)\).
We impose a boundary condition (Dirichlet or Neumann; Robin with nonnegative weight is discussed later) on
\(\partial B_0\) and simultaneously on \(\partial X_{t,0}=\Pi_t^{-1}(\partial B_0)\).



\begin{figure}[t]
\centering
\begin{tikzpicture}[
  >=Latex,
  font=\small,
  line cap=round,
  line join=round,
  boundary/.style={draw=black, line width=0.9pt},
  proj/.style={->, line width=0.9pt},
  aux/.style={draw=black!60, line width=0.35pt, dotted}, 
  punct/.style={draw=black, line width=0.7pt, fill=white},
  lab/.style={fill=white, rounded corners=1pt, inner sep=1.6pt}
]

\def\yB{0.0}    
\def\yX{4.40}   

\def\SurfacePath{
  (3.25,0.10)
    .. controls (3.05,1.30) and (1.75,1.85) .. (0.55,1.55)
    .. controls (-0.90,1.85) and (-2.70,1.20) .. (-3.10,0.15)
    .. controls (-3.55,-0.90) and (-2.45,-1.90) .. (-1.10,-1.55)
    .. controls (0.75,-1.90) and (2.95,-1.20) .. (3.25,0.10)
  -- cycle
}

\begin{scope}[shift={(0,\yB)}]
  \path[boundary, fill=white] \SurfacePath;
  \node[lab] at (0,0) {$B$};
\end{scope}

\begin{scope}[shift={(0,\yX)}]
  \path[boundary, fill=white] \SurfacePath;
  \node[lab] at (0,0) {$X_t$};
\end{scope}

\draw[proj]
  (0,\yX-1.95) -- (0,\yB+1.95)
  node[pos=0.52, right=3pt, lab] {$\Pi_t$};

\coordinate (p1) at (-1.75,\yB+0.55);
\coordinate (p2) at ( 0.55,\yB-0.70);
\coordinate (pN) at ( 1.85,\yB+0.78);

\foreach \P/\labtxt in {p1/$p_1$,p2/$p_2$,pN/$p_N$}{
  \draw[punct] (\P) circle (2.35pt);
  \node[lab, below right=1pt] at (\P) {\labtxt};
}
\node[lab] at (1.15,\yB+0.55) {$\cdots$};

\node[anchor=west, lab] (Dlabel) at (-3.95,\yB+1.45) {$D=\{p_j\}_{j=1}^N$};
\draw[->, line width=0.8pt, shorten >=2pt]
  (Dlabel.east) to[out=-10,in=170] (-2.35,\yB+0.95);

\node[anchor=west, lab] at (-3.95,\yB-1.35) {$B_{\mathrm{reg}}:=B\setminus D$};

\coordinate (bpt) at (2.05,\yB+0.10);
\fill (bpt) circle (1.3pt);
\node[lab, anchor=west] at ($(bpt)+(0.18,-0.35)$) {$b\in B_{\mathrm{reg}}$};

\coordinate (FbigC)   at (2.05,\yX+0.80);
\coordinate (FsmallC) at (2.95,\yX+0.80);

\begin{scope}[shift={(FbigC)}]
  \path[draw=black, line width=0.85pt, fill=white, even odd rule]
    (0,0) ellipse (0.78 and 0.48)
    (0,0) ellipse (0.33 and 0.20);
  \node[lab] at (0,0) {\(\mathbb{T}^2\)};
\end{scope}
\node[lab, above=3pt] at ($(FbigC)+(0,0.62)$) {regular fiber (flat)};

\draw[aux] ($(FbigC)+(0,-0.55)$) -- (bpt);

\begin{scope}[shift={(FsmallC)}]
  \path[draw=black, line width=0.85pt, fill=white, even odd rule]
    (0,0) ellipse (0.46 and 0.28)
    (0,0) ellipse (0.19 and 0.12);
\end{scope}

\draw[->, line width=0.9pt]
  ($(FbigC)+(0.90,0)$) -- ($(FsmallC)+(-0.58,0)$)
  node[midway, above, lab] {\(t\downarrow 0\)};

\node[lab, below=2pt] at ($(FsmallC)+(0,-0.45)$) {\(\mathrm{diam}(F_{b,t})\to 0\)};

\end{tikzpicture}
\caption{Schematic picture of the fibration \(\Pi_t:X_t\to B\) with finite discriminant set \(D\). Over the regular locus \(B_{\mathrm{reg}}=B\setminus D\), the fibers are flat \(2\)-tori \(\mathbb{T}^2\) that collapse as \(t\downarrow 0\).}
\label{fig:fibration-discriminant}
\end{figure}
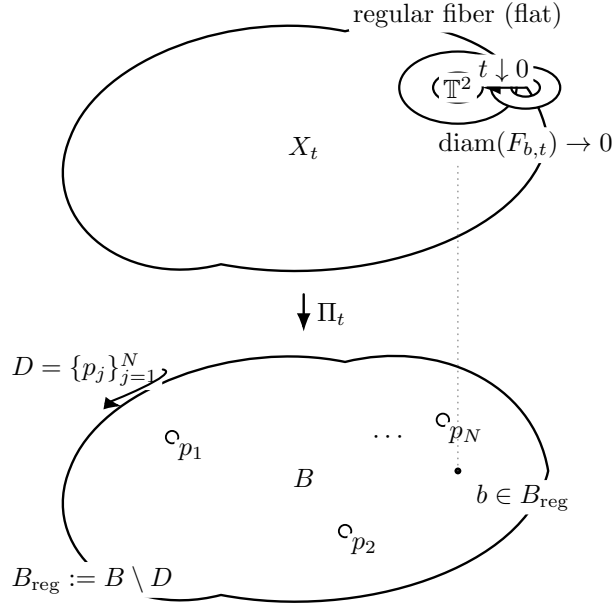

\paragraph{Semi-flat control and vertical gap (interior).}
On each \(B_0\Subset B_{\mathrm{reg}}\) we assume a semi-flat product approximation
\[
g(t)=g_B\oplus g_{F,t}+E_t
\quad\text{on }X_{t,0},
\]
where \(g_B\) is a fixed smooth metric on \(B_0\), \(g_{F,t}|_{F_{b,t}}\) is the flat metric on the fiber, and
the error \(E_t\) is exponentially small in a collapse scale \(s^*(t)\downarrow 0\):
\[
\|E_t\|_{C^1(X_{t,0})}\le \varepsilon_t,\qquad \varepsilon_t\le C e^{-c/s^*(t)}.
\]
We also assume a uniform vertical spectral gap at the collapse scale:
\[
\lambda_1(F_{b,t},g_{F,t})\ \ge\ c_0\, s^*(t)^{-2},\qquad \forall b\in B_0,
\]
where \(\lambda_1\) denotes the first nonzero eigenvalue of the fiber Laplacian.
These hypotheses are standard in semi-flat regimes arising in SYZ-type degenerations \cite{GrossWilson,OV,semiFlatRefs}.

\paragraph{Normalized lift/average.}
A key point is to compare operators on the varying spaces \(X_{t,0}\) with operators on the fixed base \(B_0\).
Using quantitative fiberwise disintegration of the volume (available under the semi-flat control),
we define the \emph{normalized lift} \(I_t:L^2(B_0)\to L^2(X_{t,0})\) and \emph{fiber-average} \(P_t:L^2(X_{t,0})\to L^2(B_0)\)
so that
\[
P_t I_t=\mathrm{Id}_{L^2(B_0)},\qquad I_t=P_t^*,
\]
and \(I_t,P_t\) are near-isometries on \(L^2\) (and compatible with horizontal \(H^1\)) with errors \(O(\varepsilon_t)\).
Heuristically, \(I_t\) identifies a base function with a fiber-constant function, normalized by the fiber area.

\subsection{Bilinear pairings and compressed semigroups}
Let \(H_{t,0}:=-\Delta_{g(t)}\) on \(L^2(X_{t,0})\) and \(H_{B,0}:=-\Delta_{g_B}\) on \(L^2(B_0)\), with the chosen boundary condition.
Denote the heat semigroups by
\[
T_t(\tau):=e^{-\tau H_{t,0}},\qquad T_B(\tau):=e^{-\tau H_{B,0}},
\]
with corresponding heat kernels \(K_t(\cdot,\cdot;\tau)\) and \(K_B(\cdot,\cdot;\tau)\).
For \(\Phi,\Psi\in C_c^\infty(B_0)\), let \(\Phi_t:=I_t\Phi\) and \(\Psi_t:=I_t\Psi\).
The bilinear pairing of \(K_t\) against fiber-constant lifts can be written equivalently as
\begin{equation}\label{eq:bilinear-via-compression}
\iint_{X_{t,0}\times X_{t,0}} K_t(x,y;\tau)\,\Phi_t(x)\Psi_t(y)\,d\mu_{g(t)}(x)d\mu_{g(t)}(y)
\ =\ \langle \Phi,\ (P_tT_t(\tau)I_t)\Psi\rangle_{L^2(B_0)}.
\end{equation}
Thus strong convergence of the compressed semigroups \(P_tT_t(\tau)I_t\) implies convergence of bilinear heat-kernel pairings.

\subsection{Discriminant regime and intrinsic conic renormalization}
The previous discussion applies on \(B_{\mathrm{reg}}\). To treat test functions that may meet the discriminant,
we incorporate the conic nature of the limiting base geometry near \(D\).

\paragraph{Conic model on the base.}
For each \(p_j\in D\), choose a wedge chart \(U_j\ni p_j\) with polar coordinates \((r,\theta)\) centered at \(p_j\),
where \(r=\mathrm{dist}_{g_B}(\cdot,p_j)\). We assume \(g_B\) is \(C^1\)-close to the flat cone
\[
g_{j}^{\mathrm{cone}}:=dr^2+\alpha_j^2 r^2 d\theta^2
\quad\text{on }U_j\setminus\{p_j\},
\]
in a quantitative sense (cf.\ Assumption~\ref{ass:conic-asymptotics}), for some cone factors \(\alpha_j>0\).
Let \(K^{\mathrm{cone}}_j(\cdot,\cdot;\tau)\) denote the Friedrichs heat kernel on \((U_j,g_{j}^{\mathrm{cone}})\),
which admits an explicit Bessel expansion and sharp Gaussian-type bounds \cite{coneHeatKernelRefs}.

\paragraph{Why renormalization is necessary.}
When \(\Phi,\Psi\) do not vanish at \(D\), the interior convergence on precompact subsets of \(B_{\mathrm{reg}}\) does not by itself
produce an intrinsic global limit: truncating away a neighborhood of \(D\) and sending the cutoff radius to zero creates a
local contribution near each conic point that must be accounted for in a canonical way.
The flat-cone model provides the correct local replacement, and leads to a cutoff-independent target functional.

\paragraph{Renormalized bilinear functional.}
Fix a smooth radial cutoff \(\eta_\rho\) supported near \(r\sim \rho\) and write
\(\Phi=\Phi^{(\rho)}+\Phi^{<\rho}\) with \(\Phi^{(\rho)}:=\eta_\rho\Phi\) and \(\mathrm{supp}(\Phi^{<\rho})\subset\{r<\rho\}\),
and similarly for \(\Psi\). Choose cutoffs \(\chi_j\in C_c^\infty(U_j)\) with \(\chi_j\equiv 1\) near \(p_j\).
We define
\begin{equation}\label{eq:Kren-intro}
\begin{split}
K_B^{\mathrm{ren}}(\Phi,\Psi;\tau)
:= \lim_{\rho\downarrow 0} \Bigg[
& \iint_{B\times B}
  K_B(b,b';\tau)\,
  \Phi^{(\rho)}(b)\Psi^{(\rho)}(b') \\
& \qquad
  d\mu_{g_B}(b)\, d\mu_{g_B}(b') \\
& + \sum_{j=1}^N
  \iint_{U_j\times U_j}
  (\chi_j\otimes\chi_j)\,
  K^{\mathrm{cone}}_j \\
& \qquad
  \Phi^{<\rho}\otimes\Psi^{<\rho}\,
  d\mu^{\otimes 2}_{g^{\mathrm{cone}}_j}
\Bigg].
\end{split}
\end{equation}
A standard conic parametrix for \(K_B\) near \(p_j\) implies that the limit exists and is independent of the choices
of \(\eta_\rho\) and \(\chi_j\); moreover if \(\mathrm{supp}\,\Phi\cup\mathrm{supp}\,\Psi\subset B_{\mathrm{reg}}\), then for $\rho$ small
one has $\Phi^{<\rho}=\Psi^{<\rho}=0$ and the renormalization is trivial, so that \(K_B^{\mathrm{ren}}(\Phi,\Psi;\tau)\) reduces to the
usual base heat-kernel pairing.

\subsection{Main result: renormalized bilinear limit across the discriminant}
We can now state the main theorem in a form suited for applications.

\begin{theorem}[Renormalized bilinear heat-kernel limit across the discriminant]\label{thm:main}
Assume the semi-flat control and the vertical spectral gap on every precompact subset of \(B_{\mathrm{reg}}\), and assume
that near each \(p_j\in D\) the base metric \(g_B\) is \(C^1\)-close to the flat cone \(g_j^{\mathrm{cone}}\) with cone factor \(\alpha_j\).
Let \(\Phi,\Psi\in C_c^\infty(B)\) and \(\tau>0\). Then, with the normalized lifts \(\Phi_t:=I_t\Phi\), \(\Psi_t:=I_t\Psi\),
\begin{equation}\label{eq:main-limit-intro}
\lim_{t\downarrow 0}\ \iint_{X_{t,0}\times X_{t,0}} K_t(x,y;\tau)\,\Phi_t(x)\Psi_t(y)\,d\mu_{g(t)}(x)d\mu_{g(t)}(y)
\ =\ K_B^{\mathrm{ren}}(\Phi,\Psi;\tau),
\end{equation}
where \(K_B^{\mathrm{ren}}\) is the canonical renormalized functional defined in \eqref{eq:Kren-intro}.
If \(\mathrm{supp}\,\Phi\cup\mathrm{supp}\,\Psi\subset B_{\mathrm{reg}}\), then \(K_B^{\mathrm{ren}}(\Phi,\Psi;\tau)\) coincides with the
usual base pairing and \eqref{eq:main-limit-intro} reduces to the interior bilinear limit.

If, in addition, the semi-flat errors are exponentially small (e.g.\ \(\|E_t\|_{C^1}\le Ce^{-c/s^*(t)}\)), then the convergence in
\eqref{eq:main-limit-intro} holds with an explicit exponential rate, uniform for \(\tau\) in compact subintervals of \((0,\infty)\).
\end{theorem}

\begin{remark}[Compressed semigroups and bilinear kernels]
On precompact \(B_0\Subset B_{\mathrm{reg}}\), the bilinear limit is equivalent to strong convergence of the compressed heat semigroups
\[
P_t e^{-\tau\Delta_{g(t)}}I_t\ \longrightarrow\ e^{-\tau\Delta_{g_B}}
\quad\text{on }L^2(B_0),
\]
and can be derived via Mosco convergence of the associated Dirichlet forms.
In the discriminant regime, the renormalized functional \eqref{eq:Kren-intro} provides a canonical way to combine interior limits
on $B_{\mathrm{reg}}$ with the flat-cone model near $D$; in particular the limiting bilinear object is independent of the auxiliary
cutoffs used to localize near the conic points.
\end{remark}

\subsection{Idea of the proof}
The argument separates naturally into an interior collapse analysis and a conic discriminant analysis.

\smallskip
\noindent\emph{(i) Interior collapse and fiberwise compression.}
On each \(B_0\Subset B_{\mathrm{reg}}\), the semi-flat control provides quantitative disintegration of \(d\mu_{g(t)}\) along the fibers and
near-isometric identifications \((I_t,P_t)\).
The vertical spectral gap yields a Poincar\'e inequality that suppresses fiber-mean-zero components and leads to a coercive
splitting of the Dirichlet energy into a fiber-constant part and a vertical part.
This gives Mosco convergence of Dirichlet forms under the identifications, hence strong resolvent/semigroup convergence
by standard form methods \cite{Attouch,Mosco1969,FukushimaOT}. The adjoint relation \(I_t=P_t^*\) and kernel representations then convert operator
limits into bilinear kernel limits as in \eqref{eq:bilinear-via-compression}, with dominated convergence justified by uniform Gaussian
bounds on bounded-geometry domains \cite{DGRefs,HeatKernelRefs,SaloffCoste}.

\smallskip
\noindent\emph{(ii) Exhaustion and boundary decoupling.}
To globalize from \(B_0\Subset B_{\mathrm{reg}}\) to \(B_{\mathrm{reg}}\), we use a smooth exhaustion and Davies--Gaffney off-diagonal estimates \cite{DGRefs,Ouhabaz}
to show that artificial boundary effects are exponentially small in the exhaustion index on fixed time windows.

\smallskip
\noindent\emph{(iii) Conic parametrix and intrinsic renormalization.}
Near each \(p_j\in D\), we compare the base heat kernel to the flat-cone heat kernel via a conic parametrix
\cite{ConeParametrixRefs,MazzeoEdge,coneHeatKernelRefs}. The remainder has an integrable defect controlled by the \(C^1\)-closeness to the cone.
This yields existence and cutoff-independence of \(K_B^{\mathrm{ren}}\) and, combined with the interior convergence on \(\{r\ge \rho\}\),
produces the renormalized limit \eqref{eq:main-limit-intro} after sending first \(t\downarrow 0\) and then \(\rho\downarrow 0\).
Quantitative bookkeeping is obtained by combining the interior convergence rate, off-diagonal (Davies--Gaffney) decoupling across a
$\rho$-buffer, and the conic parametrix remainder.

\subsection{Further consequences and organization of the paper}
The renormalized bilinear limit has a number of standard spectral consequences after compression, such as stability of localized
spectral data and convergence of localized heat traces. In order to keep the present paper focused on the discriminant mechanism and
the intrinsic renormalization, we confine ourselves to a minimal set of such corollaries.

\medskip
\noindent\textbf{Organization.}
Section~\ref{sec:identifications} develops the semi-flat disintegration and the normalized identifications.
Section~\ref{sec:mosco} proves Mosco convergence of the relevant Dirichlet forms on \(B_0\Subset B_{\mathrm{reg}}\).
Section~\ref{sec:semigroup-kernel} derives strong convergence of compressed heat semigroups and the interior bilinear kernel limit.
Section~\ref{sec:conic-ren} constructs a conic parametrix on the base, defines the cutoff-independent functional
\(K_B^{\mathrm{ren}}\), and identifies it as the global limit object across the discriminant.
Section~\ref{sec:quant-bookkeeping} records quantitative error estimates and an iterated-limit statement.
Finally, Section~\ref{sec:geometric-verification} summarizes verifiable geometric hypotheses in semi-flat elliptic fibrations that imply
the analytic assumptions used throughout.
\section{Semi-flat collapse and normalized identifications}\label{sec:identifications}

This section fixes notation and records the quantitative identifications that allow us to compare
analysis on the varying total spaces $(X_t,g(t))$ with analysis on the fixed base $(B,g_B)$.
The key output is a canonical $L^2$ identification pair
\[
I_t: L^2(B_0)\to L^2(X_{t,0}),\qquad P_t: L^2(X_{t,0})\to L^2(B_0),
\]
satisfying $P_t I_t=\mathrm{Id}$ and $I_t=P_t^*$, where $I_t$ lifts base functions to \emph{fiber-constant}
functions and $P_t$ averages along fibers with the correct normalization. In addition, $I_t$ is compatible
with the horizontal $H^1$ energy up to $O(\varepsilon_t)$ errors under the semi-flat $C^1$ control.

\subsection{Geometric setting on the regular locus}\label{subsec:setting-regular}

Let $(B,g_B)$ be a smooth Riemannian surface and let $D=\{p_1,\dots,p_N\}\subset B$ be a finite set.
We set $B_{\mathrm{reg}}:=B\setminus D$.
For each $t\in(0,t_0]$ we are given a smooth surjective submersion
\[
\Pi_t: X_t \longrightarrow B_{\mathrm{reg}}
\]
whose fibers $F_{b,t}:=\Pi_t^{-1}(b)$ are flat $2$-tori.

Fix once and for all a precompact domain $B_0\Subset B_{\mathrm{reg}}$ with smooth boundary.
We denote the restricted total space by
\[
X_{t,0}:=\Pi_t^{-1}(B_0)\Subset X_t.
\]
Throughout the paper we impose the \emph{same} boundary type on $\partial B_0$ and on
$\partial X_{t,0}=\Pi_t^{-1}(\partial B_0)$; for definiteness this section is written for
Dirichlet boundary conditions (the Neumann and nonnegatively weighted Robin cases are addressed
later, and the constructions below remain unchanged).

\subsection{Semi-flat product control and collapse scale}\label{subsec:semiflat}

We assume that on $X_{t,0}$ the metric $g(t)$ is close to a product metric with exponentially small error.
More precisely, we fix:

\begin{assumption}[Semi-flat $C^1$ product control]\label{ass:sf-C1}
There exist the following objects:
(i) a fixed smooth base metric $g_B$ on $B_0$;
(ii) a family of flat metrics $g_{F,t}$ along the fibers $F_{b,t}$;
(iii) a \emph{product reference metric} $g_t^{\Pi}:=g_B\oplus g_{F,t}$ on $X_{t,0}$ (with respect to a horizontal/vertical splitting);
(iv) a collapse scale $s(t)\downarrow 0$ as $t\downarrow 0$; and
(v) an error gauge $\varepsilon_t\ge 0$ with $\varepsilon_t\to 0$.
such that on $X_{t,0}$ we have
\begin{equation}\label{eq:sf-C1}
g(t)=g_t^{\Pi}+E_t,\qquad \|E_t\|_{C^1(X_{t,0},\,g_t^{\Pi})}\le \varepsilon_t,
\end{equation}
and the fiber diameters satisfy
\begin{equation}\label{eq:fiber-diam}
\mathrm{diam}_{g_{F,t}}(F_{b,t})\le s(t)\qquad \text{for all }b\in B_0.
\end{equation}
\end{assumption}
Such semi-flat/adiabatic product regimes arise in geometric collapse settings 
(e.g.\ elliptic K3 and abelian-fibred Calabi--Yau cases); 
see \cite{GrossWilson,TosattiAdiabatic,GrossTosattiZhang}.

\begin{assumption}[Uniform vertical spectral gap]\label{ass:vertical-gap}
There exists $c_{\mathrm{gap}}>0$ such that for all $b\in B_0$ and all $t\in(0,t_0]$,
\begin{equation}\label{eq:gap}
\lambda_1(F_{b,t},g_{F,t})\ \ge\ c_{\mathrm{gap}}\, s(t)^{-2},
\end{equation}
where $\lambda_1$ is the first nonzero eigenvalue of the (nonnegative) Laplacian on the flat torus fiber.
\end{assumption}

\begin{assumption}[Bounded geometry on the base]\label{ass:base-geometry}
The domain $(B_0,g_B)$ has positive injectivity radius and uniformly bounded curvature
(with finitely many covariant derivatives, sufficient for standard heat kernel bounds on $B_0$).
\end{assumption}

\begin{remark}[On rates]
In applications to SYZ-type collapse one often has $\varepsilon_t\le C e^{-c/s(t)}$.
In the core convergence results we only use $\varepsilon_t\to 0$; the exponential bound is only
needed when we state explicit rates.
\end{remark}

\subsection{Quantitative disintegration of volume}\label{subsec:disintegration}

Write $\mu_t:=\mu_{g(t)}$ for the Riemannian measure on $(X_{t,0},g(t))$ and $\mu_B:=\mu_{g_B}$ for the
base measure on $(B_0,g_B)$.
For each fiber $(F_{b,t},g_{F,t})$ we denote by $d\vartheta_{b,t}$ the Haar \emph{probability} measure.
By disintegration of $\mu_t$ along the submersion $\Pi_t$ (coarea) \cite{FedererGMT}, there exist measurable functions
$A_t:B_0\to(0,\infty)$ and $\rho_t:X_{t,0}\to(0,\infty)$ such that for all $f\in L^1(X_{t,0},\mu_t)$,
\begin{equation}\label{eq:disintegration-normalized}
\int_{X_{t,0}} f\,d\mu_t
=
\int_{B_0} A_t(b)\left(\int_{F_{b,t}} f\,\rho_t\,d\vartheta_{b,t}\right)d\mu_B(b),
\end{equation}
and $\rho_t$ is normalized fiberwise:
\begin{equation}\label{eq:rhot-normalization}
\int_{F_{b,t}} \rho_t\,d\vartheta_{b,t}=1\qquad\text{for all }b\in B_0.
\end{equation}
Heuristically, $A_t(b)$ is the fiber area (at base point $b$) and $\rho_t$ measures the deviation of $\mu_t$
from the product disintegration relative to $g_B\oplus g_{F,t}$.

The semi-flat $C^1$ control implies that $\rho_t$ is close to $1$ and that $A_t$ varies slowly along the base.

\begin{lemma}[Quantitative control of $(A_t,\rho_t)$]\label{lem:At-rhot-control}
Under Assumption~\ref{ass:sf-C1}, there exists $C>0$ (independent of $t$) such that on $X_{t,0}$
\begin{equation}\label{eq:rhot-control}
\|\rho_t-1\|_{L^\infty(X_{t,0})}\ \le\ C\,\varepsilon_t,
\end{equation}
and on $B_0$
\begin{equation}\label{eq:At-control}
\|\nabla_{g_B}\log A_t\|_{L^\infty(B_0)}\ \le\ C\,\varepsilon_t.
\end{equation}
\end{lemma}

\begin{proof}[Proof (sketch, with explicit dependencies)]
In a local trivialization $X_{t,0}\simeq B_0\times\mathbb T^2$ compatible with the horizontal/vertical splitting,
the product reference metric $g_t^\Pi=g_B\oplus g_{F,t}$ induces the exact disintegration
\[
d\mu_{g_t^\Pi} = A_t^\Pi(b)\, d\mu_B(b)\, d\vartheta_{b,t},
\]
with $\rho_t\equiv 1$ for the product model.
The $C^1$-smallness \eqref{eq:sf-C1} implies that the Radon--Nikodym derivative of $\mu_t$ relative to $\mu_{g_t^\Pi}$
is $1+O(\varepsilon_t)$ and that its horizontal derivatives are $O(\varepsilon_t)$.
Translating this into the normalized disintegration \eqref{eq:disintegration-normalized} gives
$\rho_t=1+O(\varepsilon_t)$ and $\nabla_{g_B}\log A_t=O(\varepsilon_t)$ uniformly on $B_0$.
A fully detailed coordinate computation is standard in semi-flat settings and is omitted.
\end{proof}

\subsection{Normalized lift and fiber-average operators}\label{subsec:It-Pt}

We can now define the canonical identifications between base and total space.

\begin{definition}[Normalized lift and average]\label{def:It-Pt}
For $v\in L^2(B_0,\mu_B)$ define the \emph{normalized lift} $I_t v\in L^2(X_{t,0},\mu_t)$ by
\begin{equation}\label{eq:It}
(I_t v)(x):=A_t(\Pi_t x)^{-1/2}\,v(\Pi_t x),
\end{equation}
which is fiber-constant by construction.
For $u\in L^2(X_{t,0},\mu_t)$ define the \emph{normalized fiber-average} $P_t u\in L^2(B_0,\mu_B)$ by
\begin{equation}\label{eq:Pt}
(P_t u)(b):=A_t(b)^{1/2}\int_{F_{b,t}} u\,\rho_t\,d\vartheta_{b,t}.
\end{equation}
\end{definition}

\begin{lemma}[Algebraic identities and adjointness]\label{lem:adjointness}
The operators $(I_t,P_t)$ satisfy
\begin{equation}\label{eq:PtIt}
P_t I_t=\mathrm{Id}_{L^2(B_0)},
\end{equation}
and $I_t=P_t^*$ as operators between $L^2(B_0,\mu_B)$ and $L^2(X_{t,0},\mu_t)$.
\end{lemma}

\begin{proof}
Let $v\in L^2(B_0)$. By \eqref{eq:Pt}, \eqref{eq:It}, and the normalization \eqref{eq:rhot-normalization},
\[
(P_t I_t v)(b)
=
A_t(b)^{1/2}\int_{F_{b,t}} A_t(b)^{-1/2}v(b)\,\rho_t\,d\vartheta_{b,t}
=
v(b).
\]
For adjointness, let $\phi\in L^2(B_0)$ and 
$u\in L^2(X_{t,0})$. 
Using \eqref{eq:disintegration-normalized}, we compute
\begin{align}
\langle I_t\phi,u\rangle_{L^2(X_{t,0})}
&= \int_{B_0} A_t(b)
   \left(
     \int_{F_{b,t}} 
       A_t(b)^{-1/2}\phi(b)\,u\,
       \rho_t\, d\vartheta_{b,t}
   \right)
   d\mu_B(b)
\nonumber\\
&= \int_{B_0} \phi(b)\,(P_tu)(b)\, d\mu_B(b)
\nonumber\\
&= \langle \phi, P_tu\rangle_{L^2(B_0)}.
\end{align}
Thus $I_t = P_t^*$.
\end{proof}

\begin{lemma}[$L^2$ isometry and contraction]\label{lem:L2-nearisometry}
For all $v\in L^2(B_0)$ and all $u\in L^2(X_{t,0})$,
\begin{equation}\label{eq:It-L2}
\|I_t v\|_{L^2(X_{t,0})}=\|v\|_{L^2(B_0)},
\end{equation}
and
\begin{equation}\label{eq:Pt-L2}
\|P_t u\|_{L^2(B_0)} \le \|u\|_{L^2(X_{t,0})}.
\end{equation}
\end{lemma}

\begin{proof}
By \eqref{eq:It}, \eqref{eq:disintegration-normalized}, and \eqref{eq:rhot-normalization},
\[
\|I_t v\|_{L^2(X_{t,0})}^2
=\int_{B_0}A_t(b)\left(\int_{F_{b,t}} A_t(b)^{-1}|v(b)|^2\,\rho_t\,d\vartheta_{b,t}\right)d\mu_B(b)
=\int_{B_0}|v(b)|^2\,d\mu_B(b).
\]
For \eqref{eq:Pt-L2}, note that $\rho_t\,d\vartheta_{b,t}$ is a probability measure on each fiber, so by Jensen
\[
|(P_tu)(b)|^2
= A_t(b)\left|\int_{F_{b,t}} u\,\rho_t\,d\vartheta_{b,t}\right|^2
\le A_t(b)\int_{F_{b,t}}|u|^2\,\rho_t\,d\vartheta_{b,t}.
\]
Integrating in $b$ and using \eqref{eq:disintegration-normalized} gives $\|P_tu\|_{L^2(B_0)}^2\le \|u\|_{L^2(X_{t,0})}^2$.
\end{proof}

\subsection{Horizontal $H^1$ compatibility for fiber-constant lifts}\label{subsec:H1-compat}

The collapsed limit is governed by the horizontal energy. Since $I_t v$ is fiber-constant, its vertical gradient vanishes.
The next lemma quantifies that its energy is close to the base energy, up to a controlled lower-order term coming from
the normalization $A_t^{-1/2}$.

\begin{lemma}[Horizontal energy of fiber-constant lifts]\label{lem:H1-lift}
Assume Lemma~\ref{lem:At-rhot-control} and the $C^1$ semi-flat control \eqref{eq:sf-C1}.
Then there exists $C>0$ independent of $t$ such that for all $v\in H^1(B_0)$,
\begin{equation}\label{eq:H1-lift}
\int_{X_{t,0}} |\nabla I_t v|_{g(t)}^2\,d\mu_t
=
\int_{B_0} |\nabla v|_{g_B}^2\,d\mu_B
\ +\ O(\varepsilon_t)\Big(\|v\|_{H^1(B_0)}^2\Big).
\end{equation}
In particular, $I_t:H^1(B_0)\to H^1(X_{t,0})$ is bounded uniformly in $t$.
\end{lemma}

\begin{proof}[Proof (outline with explicit computation)]
Work in semi-flat coordinates for $g_t^\Pi=g_B\oplus g_{F,t}$.
Since $I_t v=A_t^{-1/2}\,v\circ\Pi_t$, the vertical gradient vanishes for the product metric:
$\nabla_{F}(I_t v)\equiv 0$.
For the horizontal gradient,
\[
\nabla_B(I_t v)
= A_t^{-1/2}\,\nabla v \circ \Pi_t\ -\ \frac12 A_t^{-1/2}\,v\circ\Pi_t\,\nabla_B\log A_t.
\]
Squaring, integrating with $d\mu_t$, and using \eqref{eq:disintegration-normalized} together with
$\int_{F_{b,t}}\rho_t\,d\vartheta_{b,t}=1$ yields the identity
\[
\int_{X_{t,0}} |\nabla_B(I_t v)|_{g_B}^2\,d\mu_t
=
\int_{B_0}\Big(|\nabla v|_{g_B}^2
- \langle \nabla v, v\,\nabla_B\log A_t\rangle_{g_B}
+\frac14 |v|^2 |\nabla_B\log A_t|_{g_B}^2\Big)\,d\mu_B.
\]
The cross term is bounded by Cauchy--Schwarz and \eqref{eq:At-control}, giving an $O(\varepsilon_t)\|v\|_{H^1}^2$ contribution.
Finally, replacing $g(t)$ by $g_t^\Pi$ introduces an additional $O(\varepsilon_t)\|v\|_{H^1}^2$ error by \eqref{eq:sf-C1}.
Combining these estimates gives \eqref{eq:H1-lift}.
\end{proof}

\subsection{Fiber-constant projection (notation)}\label{subsec:projection}

We will repeatedly separate fiber-constant and fiber-oscillatory components in $L^2(X_{t,0})$.
To avoid confusion with the fibration map $\Pi_t:X_t\to B_{\mathrm{reg}}$, we denote the $L^2$ projection onto
the fiber-constant sector by $\Pi_t^{\mathrm{fc}}$.

\begin{definition}[Fiber-constant projection]\label{def:projection}
Define the projection onto the fiber-constant sector by
\[
\Pi_t^{\mathrm{fc}}:=I_tP_t: L^2(X_{t,0})\to L^2(X_{t,0}),
\]
and its complement by $(\Pi_t^{\mathrm{fc}})^\perp:=\mathrm{Id}-\Pi_t^{\mathrm{fc}}$.
\end{definition}

By Lemma~\ref{lem:adjointness}, $\Pi_t^{\mathrm{fc}}$ is self-adjoint and $(\Pi_t^{\mathrm{fc}})^2=\Pi_t^{\mathrm{fc}}$ on $L^2(X_{t,0})$.

\subsection{Section summary and next steps}\label{subsec:summary}
Assumptions~\ref{ass:sf-C1}--\ref{ass:base-geometry} yield a canonical $L^2$ identification pair $(I_t,P_t)$
with $P_tI_t=\mathrm{Id}$ and $I_t=P_t^*$, together with $L^2$ isometry/contraction
(Lemma~\ref{lem:L2-nearisometry}) and horizontal $H^1$ compatibility for fiber-constant lifts
(Lemma~\ref{lem:H1-lift}).

In the next section we introduce the Dirichlet forms and generators, use the vertical spectral gap
Assumption~\ref{ass:vertical-gap} to obtain a vertical Poincar\'e inequality for $(\Pi_t^{\mathrm{fc}})^\perp$, and then prove
Mosco convergence of Dirichlet forms on the identified spaces. This will lead to strong convergence of the
compressed heat semigroups and the interior bilinear heat-kernel limit on $B_0\Subset B_{\mathrm{reg}}$.
\section{Dirichlet forms, vertical Poincar\'e inequality, and Mosco convergence}\label{sec:mosco}

Throughout this section we fix a precompact domain $B_0\Subset B_{\mathrm{reg}}$ with smooth boundary and set
$X_{t,0}:=\Pi_t^{-1}(B_0)$. We work under Dirichlet boundary conditions on $\partial B_0$ and $\partial X_{t,0}$
(the Neumann and nonnegatively weighted Robin cases are treated later and do not affect the arguments below).

\subsection{Dirichlet forms on the base and on the total space}\label{subsec:forms}

Let $\mu_B:=\mu_{g_B}$ and $\mu_t:=\mu_{g(t)}$ denote the Riemannian measures on $B_0$ and $X_{t,0}$.
We write
\[
L^2_B:=L^2(B_0,\mu_B),\qquad L^2_t:=L^2(X_{t,0},\mu_t).
\]

\paragraph{Base Dirichlet form.}
Define the closed Dirichlet form
\begin{equation}\label{eq:EB}
E_B[v]:=\int_{B_0}|\nabla v|_{g_B}^2\,d\mu_B,\qquad \mathrm{Dom}(E_B)=H^1_0(B_0).
\end{equation}

\paragraph{Total-space Dirichlet form.}
Define
\begin{equation}\label{eq:Et}
E_t[u]:=\int_{X_{t,0}}|\nabla u|_{g(t)}^2\,d\mu_t,\qquad \mathrm{Dom}(E_t)=H^1_0(X_{t,0}).
\end{equation}
Both forms are densely defined, closed, Markovian, and generate nonnegative selfadjoint operators
$H_B$ and $H_t$ via the Friedrichs construction (we will use this in the next section); see \cite{FukushimaOT,Ouhabaz}.

\subsection{Normalized identifications and the fiber-constant projection}\label{subsec:identifications-recall}

We recall the normalized disintegration and the lift/average operators from Section~\ref{sec:identifications},
but we record explicitly the normalization used in the present paper.

\begin{assumption}[Normalized disintegration on $B_0$]\label{ass:disintegration}
There exist measurable functions $A_t:B_0\to(0,\infty)$ and $\rho_t:X_{t,0}\to(0,\infty)$ such that for all
$f\in L^1(X_{t,0},\mu_t)$,
\begin{equation}\label{eq:disintegration}
\int_{X_{t,0}} f\,d\mu_t
=
\int_{B_0}A_t(b)\left(\int_{F_{b,t}} f\,\rho_t\,d\vartheta_{b,t}\right)d\mu_B(b),
\end{equation}
where $d\vartheta_{b,t}$ is the Haar probability measure on the flat torus $(F_{b,t},g_{F,t})$, and
\begin{equation}\label{eq:rho-normalization}
\int_{F_{b,t}}\rho_t\,d\vartheta_{b,t}=1\qquad\text{for all }b\in B_0.
\end{equation}
Moreover, under the semi-flat $C^1$ control, we have the quantitative bounds
\begin{equation}\label{eq:rhoA-bounds}
\|\rho_t-1\|_{L^\infty(X_{t,0})}
+\|\nabla_{g_B}\log A_t\|_{L^\infty(B_0)}
+\|\nabla_B\rho_t\|_{L^\infty(X_{t,0})}
\ \le\ C\,\varepsilon_t,
\end{equation}
with $\varepsilon_t\downarrow 0$. Here $\nabla_B\rho_t$ denotes the horizontal (base-direction) derivative
with respect to the product splitting induced by the reference metric $g_t^\Pi=g_B\oplus g_{F,t}$ in local
semi-flat coordinates.
\end{assumption}

\begin{definition}[Normalized lift and fiber-average]\label{def:ItPt-correct}
Define bounded linear maps $I_t:L^2_B\to L^2_t$ and $P_t:L^2_t\to L^2_B$ by
\begin{equation}\label{eq:ItPt}
(I_tv)(x):=A_t(\Pi_t x)^{-1/2}\,v(\Pi_t x),\qquad
(P_tu)(b):=A_t(b)^{1/2}\int_{F_{b,t}} u\,\rho_t\,d\vartheta_{b,t}.
\end{equation}
\end{definition}

\begin{lemma}[Adjoint lock and projection]\label{lem:adjoint-lock}
The maps $(I_t,P_t)$ satisfy
\begin{equation}\label{eq:PtIt-adjoint}
P_tI_t=\mathrm{Id}_{L^2(B_0)},\qquad I_t=P_t^*.
\end{equation}
In particular, the operator
\begin{equation}\label{eq:Pi-def}
\Pi_t^{\mathrm{fc}}:=I_tP_t:L^2_t\to L^2_t
\end{equation}
is the $L^2_t$-orthogonal projection onto the closed subspace of fiber-constant functions, and
$(\Pi_t^{\mathrm{fc}})^\perp:=\mathrm{Id}-\Pi_t^{\mathrm{fc}}$ is the orthogonal projection onto its complement.
\end{lemma}

\begin{proof}
The identity $P_tI_t=\mathrm{Id}$ follows immediately from \eqref{eq:ItPt} and \eqref{eq:rho-normalization}.
For adjointness, compute using \eqref{eq:disintegration}:
for $\phi\in L^2_B$ and $u\in L^2_t$,
\[
\langle I_t\phi,u\rangle_{L^2_t}
=\int_{B_0}A_t(b)\left(\int_{F_{b,t}}A_t(b)^{-1/2}\phi(b)\,u\,\rho_t\,d\vartheta_{b,t}\right)d\mu_B(b)
=\int_{B_0}\phi(b)\, (P_tu)(b)\,d\mu_B(b)
=\langle \phi,P_tu\rangle_{L^2_B}.
\]
Thus $I_t=P_t^*$, hence $\Pi_t^{\mathrm{fc}}=I_tP_t$ is an orthogonal projection.
\end{proof}

\subsection{Vertical Poincar\'e inequality from the spectral gap}\label{subsec:vertical-poincare}

We now implement the first link in the chain:
\[
\begin{aligned}
\text{vertical spectral gap}
&\Longrightarrow \text{vertical Poincar\'e inequality} \\
&\Longrightarrow \text{suppression of nonconstant fiber modes at low energy}.
\end{aligned}
\]

Let $\nabla_F$ denote the vertical gradient with respect to the flat fiber metric $g_{F,t}$ (equivalently,
the vertical part of the product metric $g_t^\Pi=g_B\oplus g_{F,t}$).

\begin{lemma}[Fiberwise weighted Poincar\'e inequality]\label{lem:fiberwise-poincare}
Assume \eqref{eq:rhoA-bounds} and the spectral gap \eqref{eq:gap}. Then there exists $C>0$ independent of
$t$ and $b\in B_0$ such that for every $u\in H^1(F_{b,t})$,
\begin{equation}\label{eq:fiberwise-poincare}
\int_{F_{b,t}} \Big|u-\int_{F_{b,t}}u\,\rho_t\,d\vartheta_{b,t}\Big|^2 \rho_t\,d\vartheta_{b,t}
\ \le\ C\, s(t)^2 \int_{F_{b,t}}|\nabla_F u|_{g_{F,t}}^2\,\rho_t\,d\vartheta_{b,t}.
\end{equation}
\end{lemma}

\begin{proof}
Fix $b$. Let $m:=\int_{F_{b,t}}u\,d\vartheta_{b,t}$ and $m_\rho:=\int_{F_{b,t}}u\,\rho_t\,d\vartheta_{b,t}$.
Since $\int \rho_t\,d\vartheta_{b,t}=1$, we have
\[
m-m_\rho=\int_{F_{b,t}}(u-m)(1-\rho_t)\,d\vartheta_{b,t},
\]
hence by Cauchy--Schwarz and $\|\rho_t-1\|_\infty\le C\varepsilon_t$,
\[
|m-m_\rho|\le \|\rho_t-1\|_{L^2(F_{b,t},d\vartheta)}\,\|u-m\|_{L^2(F_{b,t},d\vartheta)}
\le C\varepsilon_t \|u-m\|_{L^2(F_{b,t},d\vartheta)}.
\]
Therefore,
\[
\|u-m_\rho\|_{L^2(d\vartheta)}
\le \|u-m\|_{L^2(d\vartheta)}+|m-m_\rho|
\le (1+C\varepsilon_t)\|u-m\|_{L^2(d\vartheta)}.
\]
By the standard Poincar\'e inequality on the flat torus and the gap \eqref{eq:gap},
\[
\|u-m\|_{L^2(d\vartheta)}^2 \le \lambda_1(F_{b,t},g_{F,t})^{-1}\,\|\nabla_F u\|_{L^2(d\vartheta)}^2
\le C\, s(t)^2\,\|\nabla_F u\|_{L^2(d\vartheta)}^2.
\]
Finally, since $\rho_t$ is uniformly comparable to $1$ by \eqref{eq:rhoA-bounds},
$\|\cdot\|_{L^2(\rho_t d\vartheta)}^2 \simeq \|\cdot\|_{L^2(d\vartheta)}^2$ with constants $1+O(\varepsilon_t)$,
and similarly for the gradient term; absorbing these factors yields \eqref{eq:fiberwise-poincare}.
\end{proof}

Integrating fiberwise yields the global vertical Poincar\'e inequality for the orthogonal complement of
the fiber-constant sector.

\begin{proposition}[Global vertical Poincar\'e inequality]\label{prop:vertical-poincare-global}
Under the assumptions of Lemma~\ref{lem:fiberwise-poincare}, there exists $C>0$ independent of $t$ such that
for all $u\in H^1(X_{t,0})$,
\begin{equation}\label{eq:vertical-poincare-global}
\|(\Pi_t^{\mathrm{fc}})^\perp u\|_{L^2_t}^2
\ \le\ C\, s(t)^2 \int_{X_{t,0}} |\nabla_F u|_{g_{F,t}}^2\,d\mu_t.
\end{equation}
Equivalently,
\begin{equation}\label{eq:vertical-coercivity}
\int_{X_{t,0}} |\nabla_F u|_{g_{F,t}}^2\,d\mu_t
\ \ge\ c\, s(t)^{-2}\,\|(\Pi_t^{\mathrm{fc}})^\perp u\|_{L^2_t}^2.
\end{equation}
\end{proposition}

\begin{proof}
By Lemma~\ref{lem:adjoint-lock}, $(\Pi_t^{\mathrm{fc}}u)|_{F_{b,t}}$ equals the fiberwise constant
$m_\rho(b):=\int_{F_{b,t}}u\,\rho_t\,d\vartheta_{b,t}$ (pulled back to the fiber).
Thus $((\Pi_t^{\mathrm{fc}})^\perp u)|_{F_{b,t}}=u-m_\rho(b)$. Apply Lemma~\ref{lem:fiberwise-poincare} fiberwise and integrate over $b$
using the disintegration \eqref{eq:disintegration}. This yields \eqref{eq:vertical-poincare-global}.
The coercive form \eqref{eq:vertical-coercivity} is the same inequality rewritten.
\end{proof}

\subsection{Energy comparison for lift and average}\label{subsec:energy-comparison}

To reach Mosco convergence we need two quantitative compatibilities:
(i) the lift $I_t$ is $H^1$-compatible (fiber-constants have the correct horizontal energy),
(ii) the average $P_t$ does not increase the base energy beyond the total energy.

\begin{lemma}[Energy of fiber-constant lifts]\label{lem:lift-energy}
There exists $C>0$ independent of $t$ such that for all $v\in H^1_0(B_0)$,
\begin{equation}\label{eq:lift-energy}
E_t[I_tv]
=
E_B[v]\ +\ O(\varepsilon_t)\,\|v\|_{H^1(B_0)}^2.
\end{equation}
In particular, $\lim_{t\downarrow 0}E_t[I_tv]=E_B[v]$ for each fixed $v\in H^1_0(B_0)$.
\end{lemma}

\begin{proof}
This is the $H^1$-compatibility established in Section~\ref{sec:identifications}: since $I_tv=A_t^{-1/2}\,v\circ \Pi_t$,
the vertical gradient vanishes identically and the horizontal gradient differs from $\nabla v$ only through
$\nabla\log A_t$, which is $O(\varepsilon_t)$ by \eqref{eq:rhoA-bounds}.
The $C^1$ semi-flat control ensures that the true metric $g(t)$ is uniformly close to the product metric on $X_{t,0}$,
so replacing $g(t)$ by $g_B\oplus g_{F,t}$ introduces only an $O(\varepsilon_t)$ relative error.
Collecting terms yields \eqref{eq:lift-energy}.
\end{proof}

\begin{lemma}[Energy contraction under averaging]\label{lem:avg-energy}
There exists $C>0$ independent of $t$ such that for all $u\in H^1_0(X_{t,0})$,
\begin{equation}\label{eq:avg-energy}
E_B[P_tu]\ \le\ (1+C\varepsilon_t)\,E_t[u]\ +\ C\varepsilon_t\,\|u\|_{L^2_t}^2.
\end{equation}
In particular, if $\sup_t\|u_t\|_{L^2_t}<\infty$ and $\sup_t E_t[u_t]<\infty$, then $\{P_tu_t\}$ is bounded in $H^1_0(B_0)$.
\end{lemma}

\begin{proof}
We sketch the argument, which is standard once one has the quantitative disintegration and the $C^1$ semi-flat control.

Write $v:=P_tu$. In local semi-flat coordinates on $X_{t,0}\simeq B_0\times \mathbb{T}^2$,
\[
v(b)=A_t(b)^{1/2}\int_{\mathbb{T}^2} u(b,\theta)\,\rho_t(b,\theta)\,d\vartheta.
\]
Differentiate under the integral sign along a base vector field and use the product splitting to interpret
$\nabla_B u$ as the horizontal derivative.
The leading term is $A_t^{1/2}\int \nabla_B u\ \rho_t\,d\vartheta$, which is controlled by Jensen:
\[
\Big| \int \nabla_B u\ \rho_t\,d\vartheta \Big|^2
\le \int |\nabla_B u|^2\,\rho_t\,d\vartheta.
\]
All remaining terms come from differentiating $A_t^{1/2}$ and $\rho_t$ and are therefore bounded by
$\|\nabla\log A_t\|_\infty$ and $\|\nabla_B\rho_t\|_\infty$, hence by $C\varepsilon_t$ via \eqref{eq:rhoA-bounds}.
These error terms are controlled by Cauchy--Schwarz, yielding a contribution bounded by
$C\varepsilon_t\big(\|\nabla_B u\|_{L^2_t}^2+\|u\|_{L^2_t}^2\big)$.
Finally, since $|\nabla_B u|^2\le |\nabla u|_{g(t)}^2$ up to $O(\varepsilon_t)$ errors by the $C^1$ semi-flat control,
we obtain \eqref{eq:avg-energy}.
The Dirichlet boundary condition implies $v|_{\partial B_0}=0$, hence $v\in H^1_0(B_0)$.
\end{proof}

\subsection{Mosco convergence of Dirichlet forms on identified spaces}\label{subsec:mosco}

We now prove the second link in the chain:
\[
\text{form comparison + $H^1$-compatibility of }(I_t,P_t)\Longrightarrow
\text{Mosco convergence}.
\]

\begin{definition}[Mosco convergence via $(P_t,I_t)$]\label{def:mosco}
We say that $E_t$ Mosco-converges to $E_B$ along $(P_t,I_t)$ in the sense of Mosco; see \cite{Mosco1969,Attouch,MoscoRefs}
if:
\begin{itemize}
\item[(M1) Liminf:]
whenever $u_t\in H^1_0(X_{t,0})$ satisfy $\sup_t\|u_t\|_{L^2_t}<\infty$ and $P_tu_t\to v$ in $L^2_B$,
one has $v\in H^1_0(B_0)$ and
\[
\liminf_{t\downarrow 0} E_t[u_t]\ \ge\ E_B[v];
\]
\item[(M2) Recovery:]
for every $v\in H^1_0(B_0)$ there exist $u_t\in H^1_0(X_{t,0})$ such that $P_tu_t=v$ and
\[
\limsup_{t\downarrow 0} E_t[u_t]\ \le\ E_B[v].
\]
\end{itemize}
\end{definition}

\begin{theorem}[Mosco convergence]\label{thm:mosco}
Assume the semi-flat $C^1$ control on $X_{t,0}$ (hence \eqref{eq:rhoA-bounds}) and the uniform vertical gap
Assumption~\ref{ass:vertical-gap}. Then the Dirichlet forms $E_t$ Mosco-converge to $E_B$ on $B_0$
in the sense of Definition~\ref{def:mosco} via the identification pair $(P_t,I_t)$.
\end{theorem}

\begin{proof}
\emph{Step 1: (M1) Liminf.}
Let $u_t\in H^1_0(X_{t,0})$ with $\sup_t\|u_t\|_{L^2_t}<\infty$ and $P_tu_t\to v$ in $L^2_B$.
If $\liminf_t E_t[u_t]=\infty$ there is nothing to prove, so assume $\sup_t E_t[u_t]<\infty$ along a subsequence.

By Lemma~\ref{lem:avg-energy}, the sequence $v_t:=P_tu_t$ is bounded in $H^1_0(B_0)$, hence (up to a subsequence)
$v_t\rightharpoonup v$ weakly in $H^1_0(B_0)$ and $v_t\to v$ strongly in $L^2(B_0)$.
By weak lower semicontinuity,
\begin{equation}\label{eq:lsc}
E_B[v]\ \le\ \liminf_{t\downarrow 0} E_B[v_t].
\end{equation}
On the other hand, \eqref{eq:avg-energy} gives
\[
E_B[v_t]\ \le\ (1+C\varepsilon_t)E_t[u_t] + C\varepsilon_t\|u_t\|_{L^2_t}^2.
\]
Taking $\liminf$ and using $\varepsilon_t\to 0$ and $\sup_t\|u_t\|_{L^2_t}<\infty$ yields
\begin{equation}\label{eq:EB-vt-le-Et}
\liminf_{t\downarrow 0} E_B[v_t]\ \le\ \liminf_{t\downarrow 0} E_t[u_t].
\end{equation}
Combining \eqref{eq:lsc} and \eqref{eq:EB-vt-le-Et} proves (M1).

\smallskip
\emph{Step 2: (M2) Recovery.}
Fix $v\in H^1_0(B_0)$ and set $u_t:=I_tv$.
Then $u_t\in H^1_0(X_{t,0})$ and $P_tu_t=P_tI_tv=v$ by Lemma~\ref{lem:adjoint-lock}.
By Lemma~\ref{lem:lift-energy},
\[
\limsup_{t\downarrow 0} E_t[u_t]
=
\limsup_{t\downarrow 0} E_t[I_tv]
\le E_B[v],
\]
which proves (M2).
\end{proof}

The vertical gap also yields an explicit suppression of the fiber-oscillatory component along bounded-energy sequences,
which will be used in the semigroup/kernels step.

\begin{corollary}[Vanishing of fiber-oscillatory mass at bounded energy]\label{cor:vertical-vanish}
Assume the hypotheses of Theorem~\ref{thm:mosco}. If $u_t\in H^1_0(X_{t,0})$ satisfy
$\sup_t E_t[u_t]<\infty$, then
\[
\|(\Pi_t^{\mathrm{fc}})^\perp u_t\|_{L^2_t}\ \longrightarrow\ 0
\qquad (t\downarrow 0).
\]
\end{corollary}

\begin{proof}
By Proposition~\ref{prop:vertical-poincare-global} and $|\nabla_F u|^2\le |\nabla u|_{g(t)}^2$ up to $O(\varepsilon_t)$,
\[
\|(\Pi_t^{\mathrm{fc}})^\perp u_t\|_{L^2_t}^2
\le C s(t)^2 \int_{X_{t,0}}|\nabla_F u_t|^2\,d\mu_t
\le C s(t)^2\,E_t[u_t].
\]
Since $s(t)\downarrow 0$ and $\sup_tE_t[u_t]<\infty$, the right-hand side tends to $0$.
\end{proof}

\subsection{Section summary}\label{subsec:sec3-summary}
We have proved: (i) the vertical gap yields a global vertical Poincar\'e inequality for $(\Pi_t^{\mathrm{fc}})^\perp$,
hence suppression of nonconstant fiber modes at bounded energy; (ii) the lift and average maps
$(I_t,P_t)$ are $H^1$-compatible in the precise form needed to establish Mosco convergence of the Dirichlet forms
$E_t\to E_B$ on $B_0$.

In the next section we apply the standard implication
\[
\text{Mosco convergence}\ \Longrightarrow\ \text{strong resolvent/semigroup convergence}
\]
for the associated selfadjoint operators, and then convert semigroup convergence into the
\emph{interior bilinear heat-kernel pairing limit} via kernel representation.
\section{From Mosco convergence to compressed semigroups and the interior bilinear kernel limit}\label{sec:semigroup-kernel}

In this section we close the main interior chain:
\[
\begin{aligned}
\text{Mosco convergence of } E_t
&\Longrightarrow \text{strong resolvent/semigroup convergence} \\
&\Longrightarrow \text{bilinear heat-kernel pairing limit} \\
&\qquad \text{on } B_0 \Subset B_{\mathrm{reg}}.
\end{aligned}
\]
Throughout, $B_0\Subset B_{\mathrm{reg}}$ is fixed and $X_{t,0}:=\Pi_t^{-1}(B_0)$.
We work under Dirichlet boundary conditions; the Neumann and nonnegatively weighted Robin cases
are treated later and do not affect the arguments of the present section.

\subsection{Generators and the compressed operators}\label{subsec:generators}

Let $H_t$ and $H_B$ denote the nonnegative selfadjoint operators associated with the closed forms
$E_t$ and $E_B$ from \eqref{eq:Et}--\eqref{eq:EB}, i.e.\ the Friedrichs realizations of $-\Delta_{g(t)}$ on $X_{t,0}$
and $-\Delta_{g_B}$ on $B_0$, respectively.

Write the heat semigroups
\[
T_t(\tau):=e^{-\tau H_t}\quad\text{on }L^2_t,\qquad
T_B(\tau):=e^{-\tau H_B}\quad\text{on }L^2_B,\qquad \tau>0.
\]
Both are selfadjoint, positivity preserving, and $L^2$-contractive.

We define the \emph{compressed (fiber-constant) heat operators} on the base Hilbert space $L^2_B$ by
\begin{equation}\label{eq:compressed-semigroup}
S_t(\tau):=P_t T_t(\tau) I_t \;=\; I_t^*\,T_t(\tau)\,I_t : L^2_B\to L^2_B.
\end{equation}
By Lemma~\ref{lem:adjoint-lock} and contractivity of $T_t(\tau)$, each $S_t(\tau)$ is a selfadjoint contraction on $L^2_B$:
\begin{equation}\label{eq:St-contraction}
\|S_t(\tau)\|_{L^2_B\to L^2_B}\le 1,\qquad S_t(\tau)=S_t(\tau)^*.
\end{equation}

\subsection{Mosco convergence implies strong resolvent convergence}\label{subsec:resolvent}

We first record the standard implication ``Mosco $\Rightarrow$ strong resolvent convergence'' in our identified setting.
For completeness we include a short variational proof, adapted to the connecting maps $(P_t,I_t)$.

\begin{theorem}[Strong resolvent convergence via Mosco]\label{thm:resolvent}
Assume $E_t \overset{M}{\longrightarrow} E_B$ via $(P_t,I_t)$ in the sense of Theorem~\ref{thm:mosco}.
Then for every $\alpha>0$ and every $f\in L^2_B$,
\begin{equation}\label{eq:strong-resolvent}
P_t(\alpha+H_t)^{-1}I_t f \longrightarrow (\alpha+H_B)^{-1}f\qquad\text{strongly in }L^2_B.
\end{equation}
Equivalently,
\[
P_t(\alpha+H_t)^{-1}I_t \;\xrightarrow[t\downarrow 0]{\ \mathrm{s}\ }\; (\alpha+H_B)^{-1}
\quad\text{in }\mathcal{B}(L^2_B).
\]
\end{theorem}

\begin{proof}
Fix $\alpha>0$ and $f\in L^2_B$.
Consider the strictly convex functionals
\[
\mathcal{J}_t[u] := E_t[u] + \alpha\|u\|_{L^2_t}^2 - 2\alpha\langle u, I_t f\rangle_{L^2_t},
\qquad u\in H^1_0(X_{t,0}),
\]
and
\[
\mathcal{J}_B[v] := E_B[v] + \alpha\|v\|_{L^2_B}^2 - 2\alpha\langle v, f\rangle_{L^2_B},
\qquad v\in H^1_0(B_0).
\]
By standard form theory, $\mathcal{J}_t$ and $\mathcal{J}_B$ have unique minimizers
\[
u_t=(\alpha+H_t)^{-1}I_t f \in H^1_0(X_{t,0}),\qquad v=(\alpha+H_B)^{-1}f\in H^1_0(B_0).
\]

\emph{Step 1: compactness of the pushed-forward minimizers.}
Set $v_t:=P_tu_t\in L^2_B$. Since $u_t$ minimizes $\mathcal{J}_t$, we have
\[
\mathcal{J}_t[u_t] \le \mathcal{J}_t[I_t v]
= E_t[I_tv]+\alpha\|I_tv\|_{L^2_t}^2 - 2\alpha\langle I_t v, I_t f\rangle_{L^2_t}.
\]
Using $I_t=P_t^*$ and $P_tI_t=\mathrm{Id}$, we have the exact identities
\[
\|I_tv\|_{L^2_t}^2=\langle I_tv,I_tv\rangle_{L^2_t}=\langle v,P_tI_tv\rangle_{L^2_B}=\|v\|_{L^2_B}^2,
\]
and
\[
\langle I_tv,I_tf\rangle_{L^2_t}=\langle v,P_tI_tf\rangle_{L^2_B}=\langle v,f\rangle_{L^2_B}.
\]
Moreover, Lemma~\ref{lem:lift-energy} gives $E_t[I_tv]=E_B[v]+o(1)$.
Hence $\sup_t \mathcal{J}_t[u_t]<\infty$, which implies $\sup_t\|u_t\|_{L^2_t}<\infty$ and $\sup_t E_t[u_t]<\infty$.
By Lemma~\ref{lem:avg-energy}, $\{v_t\}$ is bounded in $H^1_0(B_0)$, thus (up to a subsequence)
$v_t\rightharpoonup v_\infty$ weakly in $H^1_0$ and $v_t\to v_\infty$ strongly in $L^2_B$.

\emph{Step 2: identify the limit as the base resolvent.}
By Mosco (M1), applied to $u_t$ with $P_tu_t=v_t\to v_\infty$ in $L^2_B$,
\[
\liminf_{t\downarrow 0} E_t[u_t]\ge E_B[v_\infty].
\]
Also, using $I_t=P_t^*$ and $P_tu_t=v_t$,
\[
\langle u_t, I_t f\rangle_{L^2_t}=\langle P_tu_t,f\rangle_{L^2_B}=\langle v_t,f\rangle_{L^2_B}\to \langle v_\infty,f\rangle_{L^2_B}.
\]
Finally, since $\|P_t\|=\|I_t\|=1$ (as $I_t$ is an isometry and $P_t=I_t^*$), we have
\[
\|u_t\|_{L^2_t}\ge \|P_tu_t\|_{L^2_B}=\|v_t\|_{L^2_B},
\]
hence
\[
\liminf_{t\downarrow 0}\|u_t\|_{L^2_t}^2 \ge \lim_{t\downarrow 0}\|v_t\|_{L^2_B}^2 = \|v_\infty\|_{L^2_B}^2.
\]
Putting these together yields
\[
\liminf_{t\downarrow 0}\mathcal{J}_t[u_t]\ \ge\ \mathcal{J}_B[v_\infty].
\]
On the other hand, the recovery sequence $I_t v$ (Mosco (M2)) implies
\[
\limsup_{t\downarrow 0}\inf \mathcal{J}_t \ \le\ \limsup_{t\downarrow 0}\mathcal{J}_t[I_tv] \ =\ \mathcal{J}_B[v].
\]
Therefore $\mathcal{J}_B[v_\infty]\le \mathcal{J}_B[v]$. By uniqueness of the minimizer of $\mathcal{J}_B$,
we have $v_\infty=v$.
Since the limit is unique, the whole family $v_t=P_tu_t$ converges to $v$ in $L^2_B$, proving \eqref{eq:strong-resolvent}.
\end{proof}

\subsection{Strong convergence of compressed heat operators}\label{subsec:semigroup}

We now pass from resolvents to heat operators.

\begin{theorem}[Strong convergence of compressed heat operators]\label{thm:semigroup}
Assume the hypotheses of Theorem~\ref{thm:resolvent}. Then for every $\tau>0$,
\begin{equation}\label{eq:strong-semigroup}
S_t(\tau)=P_t e^{-\tau H_t} I_t \longrightarrow e^{-\tau H_B}=T_B(\tau)
\qquad\text{strongly in }\mathcal{B}(L^2_B).
\end{equation}
Moreover, the convergence is uniform for $\tau$ in compact subsets of $(0,\infty)$.
\end{theorem}

\begin{proof}
For nonnegative selfadjoint operators, strong resolvent convergence implies strong convergence of
the bounded functional calculus for every bounded continuous function vanishing at infinity \cite{Kato,ReedSimonI}.
In particular, for each fixed $\tau>0$ the function $x\mapsto e^{-\tau x}$ yields
\[
P_t e^{-\tau H_t} I_t \to e^{-\tau H_B}\quad\text{strongly on }L^2_B.
\]
A direct proof can be obtained by approximating $e^{-\tau x}$ uniformly on $[0,\infty)$ by rational
functions in $(\alpha+x)^{-1}$ (or by Euler approximants $(1+\tau x/n)^{-n}$) and using
Theorem~\ref{thm:resolvent} together with uniform operator bounds \eqref{eq:St-contraction}.
Uniformity on compact $\tau$-intervals follows from uniform approximation.
\end{proof}

\subsection{Heat kernels and bilinearization}\label{subsec:kernels}

On the precompact domains $B_0$ and $X_{t,0}$ (with smooth boundary), the resolvents are compact and the
heat operators admit integral kernels.

\begin{lemma}[Kernel representation]\label{lem:kernel-representation}
For each $\tau>0$ there exist symmetric kernels
\[
K_t(\cdot,\cdot;\tau)\in L^2(X_{t,0}\times X_{t,0}),\qquad
K_B(\cdot,\cdot;\tau)\in L^2(B_0\times B_0),
\]
such that for all $F,G\in L^2_t$ and all $\phi,\psi\in L^2_B$,
\begin{align}
\langle F,T_t(\tau)G\rangle_{L^2_t}
&=\int_{X_{t,0}\times X_{t,0}} K_t(x,y;\tau)\,F(x)\,G(y)\,d\mu_t(x)\,d\mu_t(y),\label{eq:kernel-total}\\
\langle \phi,T_B(\tau)\psi\rangle_{L^2_B}
&=\int_{B_0\times B_0} K_B(b,b';\tau)\,\phi(b)\,\psi(b')\,d\mu_B(b)\,d\mu_B(b').\label{eq:kernel-base}
\end{align}
\end{lemma}

\begin{proof}[Proof sketch]
On a precompact domain with Dirichlet boundary, the embedding $H^1_0\hookrightarrow L^2$ is compact \cite{AdamsFournier},
hence $(\alpha+H)^{-1}$ is compact and $H$ has discrete spectrum with an $L^2$-orthonormal basis of eigenfunctions.
The kernel is given by the usual spectral expansion $K(x,y;\tau)=\sum_k e^{-\tau\lambda_k}\varphi_k(x)\varphi_k(y)$,
converging in $L^2$ on the product domain; see \cite{DGRefs,HeatKernelRefs}
\end{proof}

The compressed operator $S_t(\tau)$ admits an exact bilinear kernel expression.

\begin{proposition}[Bilinearization identity]\label{prop:bilinearization}
For every $\tau>0$ and every $\Phi,\Psi\in L^2_B$,
\begin{equation}\label{eq:bilinearization-operator}
\langle \Phi, S_t(\tau)\Psi\rangle_{L^2_B}
=\langle I_t\Phi, T_t(\tau) I_t\Psi\rangle_{L^2_t}.
\end{equation}
If the kernels of Lemma~\ref{lem:kernel-representation} are used, then
\begin{equation}\label{eq:bilinearization-kernel}
\langle \Phi, S_t(\tau)\Psi\rangle_{L^2_B}
=
\int_{X_{t,0}\times X_{t,0}} K_t(x,y;\tau)\,(I_t\Phi)(x)\,(I_t\Psi)(y)\,d\mu_t(x)\,d\mu_t(y),
\end{equation}
and similarly
\begin{equation}\label{eq:bilinearization-base}
\langle \Phi, T_B(\tau)\Psi\rangle_{L^2_B}
=
\int_{B_0\times B_0} K_B(b,b';\tau)\,\Phi(b)\,\Psi(b')\,d\mu_B(b)\,d\mu_B(b').
\end{equation}
\end{proposition}

\begin{proof}
By Lemma~\ref{lem:adjoint-lock}, $I_t=P_t^*$, hence
\[
\langle \Phi, S_t(\tau)\Psi\rangle_{L^2_B}
=\langle \Phi, P_tT_t(\tau)I_t\Psi\rangle_{L^2_B}
=\langle I_t\Phi, T_t(\tau)I_t\Psi\rangle_{L^2_t},
\]
which is \eqref{eq:bilinearization-operator}. The kernel identities follow from \eqref{eq:kernel-total} and \eqref{eq:kernel-base}.
\end{proof}

\subsection{Interior bilinear heat-kernel pairing limit}\label{subsec:bilinear-limit}

We can now state the interior bilinear limit as a direct corollary of Theorem~\ref{thm:semigroup}.

\begin{theorem}[Interior bilinear heat-kernel pairing limit]\label{thm:bilinear-interior}
Assume the hypotheses of Theorem~\ref{thm:semigroup}, and let $B_0\Subset B_{\mathrm{reg}}$.
Then for every $\tau>0$ and every $\Phi,\Psi\in L^2(B_0)$,
\begin{equation}\label{eq:bilinear-limit-L2}
\int_{X_{t,0}\times X_{t,0}} K_t(x,y;\tau)\,(I_t\Phi)(x)\,(I_t\Psi)(y)\,d\mu_t(x)\,d\mu_t(y)
\ \longrightarrow\
\int_{B_0\times B_0} K_B(b,b';\tau)\,\Phi(b)\,\Psi(b')\,d\mu_B(b)\,d\mu_B(b').
\end{equation}
In particular, the convergence holds for all $\Phi,\Psi\in C_c^\infty(B_0)$.
\end{theorem}

\begin{proof}
By Proposition~\ref{prop:bilinearization}, the left-hand side of \eqref{eq:bilinear-limit-L2} equals
$\langle \Phi,S_t(\tau)\Psi\rangle_{L^2_B}$ and the right-hand side equals $\langle \Phi,T_B(\tau)\Psi\rangle_{L^2_B}$.
By Theorem~\ref{thm:semigroup}, $S_t(\tau)\Psi\to T_B(\tau)\Psi$ strongly in $L^2_B$; testing against $\Phi$
gives \eqref{eq:bilinear-limit-L2}.
\end{proof}

\subsection{Strong convergence versus bilinear pairing limits (a rigorous equivalence)}\label{subsec:equivalence}

For later use (and to match the ``compressed semigroup $\Leftrightarrow$ bilinear kernel'' viewpoint precisely),
we record a rigorous equivalence at the level of the \emph{whole family} $\{\tau>0\}$.
The only extra input needed beyond bilinear matrix-coefficient convergence is an ``asymptotic semigroup'' property
for $S_t(\tau)$, which follows from the vertical gap and smoothing.

\begin{lemma}[Energy smoothing for the heat flow]\label{lem:smoothing}
For every $\sigma>0$ and every $G\in L^2_t$,
\begin{equation}\label{eq:smoothing}
E_t[T_t(\sigma)G]\ \le\ \frac{C}{\sigma}\,\|G\|_{L^2_t}^2,
\end{equation}
where $C>0$ is universal.
\end{lemma}

\begin{proof}
By spectral calculus,
\[
E_t[T_t(\sigma)G]=\|H_t^{1/2}e^{-\sigma H_t}G\|_{L^2_t}^2
=\int_0^\infty \lambda\,e^{-2\sigma\lambda}\,d\|E_\lambda G\|^2
\le \Big(\sup_{\lambda\ge 0}\lambda e^{-2\sigma\lambda}\Big)\|G\|_{L^2_t}^2
\le \frac{C}{\sigma}\|G\|_{L^2_t}^2.
\]
\end{proof}

\begin{lemma}[Leakage estimate from fiber-constant data]\label{lem:leakage}
Assume the vertical Poincar\'e inequality \eqref{eq:vertical-poincare-global} from Proposition~\ref{prop:vertical-poincare-global}.
Then for every $\sigma>0$ and every $f\in L^2_B$,
\begin{equation}\label{eq:leakage}
\|(\Pi_t^{\mathrm{fc}})^\perp\, T_t(\sigma) I_tf\|_{L^2_t}\ \le\ C\,\frac{s(t)}{\sqrt{\sigma}}\;\|f\|_{L^2_B},
\end{equation}
with $C$ independent of $t$.
\end{lemma}

\begin{proof}
Apply \eqref{eq:vertical-poincare-global} to $u=T_t(\sigma)I_tf$ and use $|\nabla_F u|^2\le |\nabla u|^2$:
\[
\|(\Pi_t^{\mathrm{fc}})^\perp u\|_{L^2_t}^2 \le C s(t)^2 \int_{X_{t,0}}|\nabla u|^2\,d\mu_t
= C s(t)^2\,E_t[u].
\]
Then invoke Lemma~\ref{lem:smoothing} and the $L^2$-isometry of $I_t$ to get
$E_t[u]\le C\sigma^{-1}\|I_tf\|_{L^2_t}^2=C\sigma^{-1}\|f\|_{L^2_B}^2$.
\end{proof}

\begin{proposition}[Asymptotic semigroup property of $S_t(\tau)$]\label{prop:asymptotic-semigroup}
For all $\tau,\sigma>0$,
\begin{equation}\label{eq:asymptotic-semigroup}
\|S_t(\tau)S_t(\sigma)-S_t(\tau+\sigma)\|_{L^2_B\to L^2_B}
\ \le\ C\,\frac{s(t)}{\sqrt{\sigma}}.
\end{equation}
In particular, for each fixed $\tau>0$, $\|S_t(\tau)^2-S_t(2\tau)\|_{op}\to 0$ as $t\downarrow 0$.
\end{proposition}

\begin{proof}
Using $P_tI_t=\mathrm{Id}$ and $\Pi_t^{\mathrm{fc}}=I_tP_t$,
\[
S_t(\tau)S_t(\sigma)-S_t(\tau+\sigma)
= P_tT_t(\tau)I_tP_tT_t(\sigma)I_t - P_tT_t(\tau)T_t(\sigma)I_t
= -P_tT_t(\tau)(\Pi_t^{\mathrm{fc}})^\perp T_t(\sigma)I_t.
\]
Taking operator norms and using $\|P_t\|\le 1$, $\|T_t(\tau)\|\le 1$ and Lemma~\ref{lem:leakage} gives \eqref{eq:asymptotic-semigroup}.
\end{proof}

We can now state a rigorous ``strong $\Leftrightarrow$ bilinear kernel'' equivalence for the \emph{full time family}.

\begin{theorem}[Strong compressed limit $\Longleftrightarrow$ bilinear kernel limits]\label{thm:equivalence-strong-bilinear}
Assume Proposition~\ref{prop:asymptotic-semigroup}.
Then the following two statements are equivalent:

\begin{itemize}
\item[(S)] For every $\tau>0$, $S_t(\tau)\to T_B(\tau)$ strongly on $L^2_B$.

\item[(K)] For every $\tau>0$ and every $\Phi,\Psi\in L^2(B_0)$,
\begin{equation}\label{eq:K-statement}
\int_{X_{t,0}\times X_{t,0}} K_t(x,y;\tau)\,(I_t\Phi)(x)\,(I_t\Psi)(y)\,d\mu_t(x)\,d\mu_t(y)
\ \longrightarrow\
\int_{B_0\times B_0} K_B(b,b';\tau)\,\Phi(b)\,\Psi(b')\,d\mu_B(b)\,d\mu_B(b').
\end{equation}
\end{itemize}

\end{theorem}

\begin{proof}
\emph{(S)$\Rightarrow$(K)} is exactly Theorem~\ref{thm:bilinear-interior}.

\emph{(K)$\Rightarrow$(S).}
Fix $\tau>0$ and $f\in L^2_B$. Set $x_t:=S_t(\tau)f\in L^2_B$.
By Proposition~\ref{prop:bilinearization}, (K) implies that for every $\Phi\in L^2_B$,
$\langle \Phi, x_t\rangle \to \langle \Phi, T_B(\tau)f\rangle$, hence $x_t\rightharpoonup T_B(\tau)f$ weakly in $L^2_B$.

To upgrade to strong convergence, it suffices to show $\|x_t\|\to \|T_B(\tau)f\|$.
Since $S_t(\tau)$ is selfadjoint,
\[
\|x_t\|_{L^2_B}^2=\langle S_t(\tau)f, S_t(\tau)f\rangle=\langle f, S_t(\tau)^2 f\rangle.
\]
By Proposition~\ref{prop:asymptotic-semigroup},
\[
\langle f,S_t(\tau)^2 f\rangle = \langle f,S_t(2\tau)f\rangle + o(1)\|f\|^2.
\]
Applying (K) at time $2\tau$ yields $\langle f,S_t(2\tau)f\rangle\to \langle f,T_B(2\tau)f\rangle$.
Finally, since $T_B(\tau)$ is a selfadjoint semigroup,
\[
\langle f,T_B(2\tau)f\rangle = \langle T_B(\tau)f, T_B(\tau)f\rangle = \|T_B(\tau)f\|^2.
\]
Thus $\|x_t\|^2\to \|T_B(\tau)f\|^2$, and together with weak convergence this implies $x_t\to T_B(\tau)f$ strongly.
Since $f$ was arbitrary, we obtain (S).
\end{proof}

\subsection{Section summary}\label{subsec:sec4-summary}

We have shown the following:
(i) Mosco convergence of $E_t$ to $E_B$ implies strong resolvent convergence \eqref{eq:strong-resolvent}, 
hence strong convergence of the compressed heat operators $S_t(\tau)\to T_B(\tau)$ for every $\tau>0$;
(ii) by kernel representation and bilinearization, strong convergence yields the interior bilinear heat-kernel pairing limit 
\eqref{eq:bilinear-limit-L2} on $B_0\Subset B_{\mathrm{reg}}$; and
(iii) conversely, the family of bilinear kernel limits for all $\tau>0$ implies strong convergence of $S_t(\tau)$, 
once the asymptotic semigroup property \eqref{eq:asymptotic-semigroup} is observed, a property that follows from 
the vertical gap and smoothing.

In the next section we move beyond $B_{\mathrm{reg}}$:
near each discriminant point $p_j\in D$ the base metric is modeled by a flat cone, and the bilinear pairing
requires an intrinsic conic renormalization.
\section{Conic parametrix on the base and the renormalized bilinear functional}\label{sec:conic-ren}

Let $B$ be a smooth connected surface, and let
\[
D=\{p_1,\dots,p_N\}\subset B
\]
be a finite discriminant set. Set $B_{\mathrm{reg}}:=B\setminus D$ and let $g_B$ be a smooth Riemannian metric on
$B_{\mathrm{reg}}$. We assume that $g_B$ has \emph{isolated conic asymptotics} at each $p_j$ in the sense below.

This section has four goals:
\begin{enumerate}
\item Build wedge charts near $D$ and formulate a $C^1$ conic approximation of $g_B$.
\item Introduce the flat-cone heat kernel and construct a conic parametrix for the base heat kernel with weighted remainders.
\item Define a \emph{cutoff-independent} renormalized bilinear functional $K^{\mathrm{ren}}_B(\Phi,\Psi;\tau)$.
\item Identify the collapsed total-space bilinear limits with $K^{\mathrm{ren}}_B$ for each fixed $\tau>0$.
\end{enumerate}

\subsection{Wedge charts and $C^1$ conic asymptotics of $g_B$ near $D$}\label{subsec:wedge-cone}

Fix $p_j\in D$. Let $U_j\subset B$ be a coordinate neighborhood of $p_j$, and choose polar coordinates
\[
(r,\theta)\in(0,r_0)\times \mathbb{S}^1\qquad \text{on }U_j\setminus\{p_j\}
\]
centered at $p_j$ (so that $r\downarrow 0$ corresponds to approaching $p_j$). We assume:

\begin{assumption}[$C^1$ conic asymptotics]\label{ass:conic-asymptotics}
There exist parameters $\alpha_j>0$ (cone angles $2\pi\alpha_j$) and $\beta\in(0,1]$ such that on $U_j\setminus\{p_j\}$,
\begin{equation}\label{eq:conic-metric-expansion}
g_B = g_{j}^{\mathrm{cone}} + q_j,\qquad
g_{j}^{\mathrm{cone}}:=dr^2+\alpha_j^2 r^2\,d\theta^2,
\end{equation}
where the $2$-tensor $q_j$ satisfies the weighted $C^1$ bounds
\begin{equation}\label{eq:C1-cone-control}
|q_j|_{g^{\mathrm{cone}}_j} \le C r^{\beta},\qquad
|\nabla^{\mathrm{cone}} q_j|_{g^{\mathrm{cone}}_j} \le C r^{\beta-1},
\end{equation}
for some $C>0$ independent of $(r,\theta)$.
In particular, the Riemannian measures satisfy
\begin{equation}\label{eq:measure-compare}
d\mu_{g_B} = (1+O(r^\beta))\, d\mu_{g^{\mathrm{cone}}_j},\qquad d\mu_{g^{\mathrm{cone}}_j} = \alpha_j r\,dr\,d\theta,
\end{equation}
and the distances induced by $g_B$ and $g^{\mathrm{cone}}_j$ are locally bi-Lipschitz equivalent on $U_j\setminus\{p_j\}$.
\end{assumption}

\begin{remark}\label{rem:cone-model-scope}
Assumption~\ref{ass:conic-asymptotics} is precisely what we need for a \emph{stable conic heat parametrix}
with remainders controlled by $(r+r')^\beta$ weights.
It does not require $g_B$ to extend smoothly across $p_j$.
\end{remark}
See Figure~\ref{fig:wedge-cone-model} for the wedge chart and the inner/outer decomposition used in the renormalization.


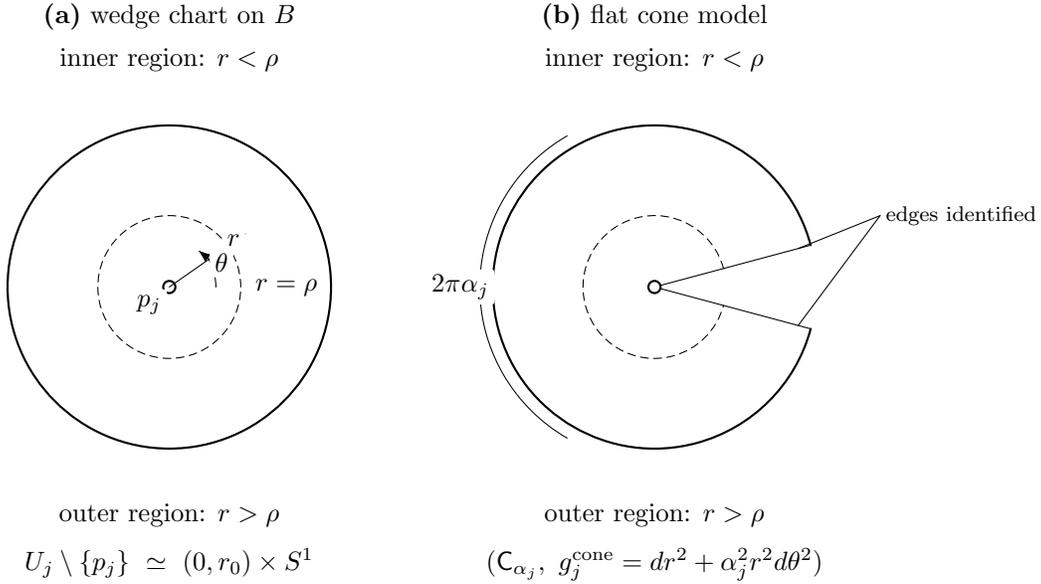
\begin{figure}[t]
\centering
\begin{tikzpicture}[
  >=Latex,
  font=\small,
  line cap=round,
  line join=round,
  lab/.style={fill=white, rounded corners=1pt, inner sep=2pt},
  boundary/.style={thick},
  cut/.style={thin},
  inner/.style={densely dashed, thin},
  auxarrow/.style={->, thin},
  leader/.style={thin} 
]

\def\Rout{2.15}          
\def\Rin{0.95}           
\def\sep{6.45}           
\def\labdist{0.65}       
\def\titledist{0.55}     
\def\belowdist{0.60}     

\def\halfcut{15}         

\def\Rr{1.25}            
\def\Rtheta{0.62}        

\begin{scope}[shift={(0,0)}]

  \draw[boundary] (0,0) circle (\Rout);
  \draw[inner]    (0,0) circle (\Rin);

  \draw[boundary, fill=white] (0,0) circle (2.3pt);
  \node[lab, below left=1pt] at (0,0) {$p_j$};

  \draw[auxarrow] (0,0) -- (35:\Rr)
    node[pos=0.70, lab, above right=-1pt] {$r$};

  \draw[auxarrow] (\Rtheta,0) arc (0:52:\Rtheta)
    node[pos=0.60, lab, right] {$\theta$};

  \node[lab, anchor=west] at (\Rin+0.12,0) {$r=\rho$};

  \node[lab, anchor=south] at (0,\Rout+\labdist+\titledist)
    {\textbf{(a)} wedge chart on $B$};

  \node[lab, anchor=south] at (0,\Rout+\labdist)
    {inner region: $r<\rho$};
  \node[lab, anchor=north] at (0,-\Rout-\labdist)
    {outer region: $r>\rho$};

  \node[lab, anchor=north] at (0,-\Rout-\labdist-\belowdist)
    {$U_j\setminus\{p_j\}\ \simeq\ (0,r_0)\times S^1$};

\end{scope}

\begin{scope}[shift={(\sep,0)}]

  \draw[boundary] (\halfcut:\Rout) arc (\halfcut:360-\halfcut:\Rout);

  \draw[cut] (0,0) -- (\halfcut:\Rout);
  \draw[cut] (0,0) -- (360-\halfcut:\Rout);

  \draw[inner] (\halfcut:\Rin) arc (\halfcut:360-\halfcut:\Rin);

  \draw[boundary, fill=white] (0,0) circle (2.3pt);

  \node[lab, anchor=south] at (0,\Rout+\labdist+\titledist)
    {\textbf{(b)} flat cone model};

  \node[lab, anchor=south] at (0,\Rout+\labdist)
    {inner region: $r<\rho$};
  \node[lab, anchor=north] at (0,-\Rout-\labdist)
    {outer region: $r>\rho$};

  \draw[cut] (120:{1.08*\Rout}) arc (120:240:{1.08*\Rout});
  \node[lab] at (180:{1.20*\Rout}) {$2\pi\alpha_j$};

  \node[lab, anchor=west] (edgelab) at (\Rout+0.85,0.95) {\scriptsize edges identified};
  \draw[leader] (edgelab.west) -- (\halfcut:{0.92*\Rout});
  \draw[leader] (edgelab.west) -- (360-\halfcut:{0.92*\Rout});

  \node[lab, anchor=north] at (0,-\Rout-\labdist-\belowdist)
    {$(\mathsf{C}_{\alpha_j},\ g^{\mathrm{cone}}_j=dr^2+\alpha_j^2 r^2 d\theta^2)$};

\end{scope}

\end{tikzpicture}

\vspace{0.35em}

\caption{Local wedge chart near a discriminant point \(p_j\in D\) and the corresponding flat-cone model.
Panel (a) shows polar coordinates \((r,\theta)\) on \(U_j\setminus\{p_j\}\subset B_{\mathrm{reg}}\), together with the
inner region \(r<\rho\) and outer region \(r>\rho\) used in the renormalization.
Panel (b) depicts the model cone \(\mathsf{C}_{\alpha_j}\) with cone angle \(2\pi\alpha_j\), for which \(g_B\) is \(C^1\)-close in Assumption~5.1.}
\label{fig:wedge-cone-model}
\end{figure}

\subsection{The flat-cone heat kernel and basic bounds}\label{subsec:cone-kernel}

Let $\mathsf{C}_{\alpha_j}:=((0,\infty)\times \mathbb{S}^1,\, g^{\mathrm{cone}}_j)$ denote the flat cone of angle $2\pi\alpha_j$.
Let $H^{\mathrm{cone}}_j$ be the Friedrichs realization of $-\Delta_{g^{\mathrm{cone}}_j}$ on $L^2(\mathsf{C}_{\alpha_j},d\mu_{g^{\mathrm{cone}}_j})$.
Denote by $K^{\mathrm{cone}}_j(\cdot,\cdot;\tau)$ its heat kernel.

One may write $K^{\mathrm{cone}}_j$ explicitly via Bessel functions; for our purposes we only need:
positivity, symmetry, Gaussian domination, and scale-invariant derivative bounds; see 
\cite{coneHeatKernelRefs,ConeParametrixRefs}.

\begin{lemma}[Model cone kernel estimates]\label{lem:cone-kernel-bounds}
For each $j$ and each $\tau_0>0$ there exist constants $C,c>0$ (depending only on $\alpha_j$ and $\tau_0$) such that for all
$0<\tau\le \tau_0$ and all $(r,\theta),(r',\theta')\in (0,r_0)\times\mathbb{S}^1$,
\begin{equation}\label{eq:cone-gaussian}
0 \le K^{\mathrm{cone}}_j((r,\theta),(r',\theta');\tau)
\le
\frac{C}{\tau}\exp\!\Big(-\frac{d_{\mathrm{cone}}((r,\theta),(r',\theta'))^2}{c\,\tau}\Big),
\end{equation}
and
\begin{equation}\label{eq:cone-derivative}
|\partial_r K^{\mathrm{cone}}_j| + |\partial_{r'} K^{\mathrm{cone}}_j|
\le
\frac{C}{\tau^{3/2}}\exp\!\Big(-\frac{d_{\mathrm{cone}}((r,\theta),(r',\theta'))^2}{c\,\tau}\Big).
\end{equation}
\end{lemma}

\begin{remark}\label{rem:cone-kernel-refs}
Such bounds are standard for heat kernels on flat cones; they follow from the explicit Bessel expansion
and from general heat kernel estimates on conic manifolds (e.g.\ the parametrix calculus).
\end{remark}

\subsection{Conic parametrix for the base heat kernel with weighted remainders}\label{subsec:base-parametrix}

Let $H_B$ denote the Friedrichs realization of $-\Delta_{g_B}$ on $L^2(B_{\mathrm{reg}},d\mu_{g_B})$,
and let $K_B(\cdot,\cdot;\tau)$ be its heat kernel on $B_{\mathrm{reg}}\times B_{\mathrm{reg}}$.

Fix cutoffs $\chi_j,\widetilde\chi_j\in C_c^\infty(U_j)$ such that
\[
\chi_j\equiv 1 \ \text{on a neighborhood of }p_j,\qquad \widetilde\chi_j\equiv 1 \ \text{on }\mathrm{supp}\,\chi_j.
\]
We view $K^{\mathrm{cone}}_j$ as a kernel on $U_j\setminus\{p_j\}$ via the wedge coordinates.

\begin{proposition}[Local conic parametrix with $(r+r')^\beta$ remainder]\label{prop:conic-parametrix}
Assume Assumption~\ref{ass:conic-asymptotics}. A local conic heat parametrix with weighted remainders can be constructed by standard edge methods; 
see \cite{ConeParametrixRefs,MazzeoEdge,coneHeatKernelRefs}.
Then there exist $\tau_0\in(0,1]$ and kernels
$R_j(\cdot,\cdot;\tau)$ on $(U_j\setminus\{p_j\})\times (U_j\setminus\{p_j\})$ such that for all $0<\tau\le\tau_0$ and all
$b,b'\in U_j\setminus\{p_j\}$,
\begin{equation}\label{eq:KB-parametrix}
K_B(b,b';\tau)
=
\chi_j(b)\chi_j(b')\,K^{\mathrm{cone}}_j(b,b';\tau)
+
R_j(b,b';\tau),
\end{equation}
and the remainder satisfies the weighted bound
\begin{equation}\label{eq:Rj-bound}
|R_j(b,b';\tau)|
\le
\frac{C}{\tau}\,(r(b)+r(b'))^{\beta}\,
\exp\!\Big(-\frac{d_{g_B}(b,b')^2}{C\,\tau}\Big),
\qquad 0<\tau\le\tau_0,
\end{equation}
where $r(\cdot)$ is the wedge radial variable around $p_j$.
\end{proposition}

\begin{proof}[Proof sketch]
Let $K^{(0)}:=\widetilde\chi_j\otimes\widetilde\chi_j\,K^{\mathrm{cone}}_j$.
Using \eqref{eq:conic-metric-expansion}--\eqref{eq:C1-cone-control}, the operator difference
$H_B-H^{\mathrm{cone}}_j$ has coefficients of size $O(r^\beta)$ and first derivatives of size $O(r^{\beta-1})$
in the wedge coordinates. A Duhamel parametrix construction yields
\[
(\partial_\tau+H_B)K^{(0)}=\delta + \mathcal{E},
\]
where $\mathcal{E}$ is supported where $\nabla\widetilde\chi_j\neq 0$ or where $q_j\neq 0$,
and satisfies Gaussian domination with a prefactor $(r+r')^\beta\tau^{-1}$ thanks to
\eqref{eq:cone-gaussian}--\eqref{eq:cone-derivative} and \eqref{eq:C1-cone-control}.
Setting
\[
R_j:=\int_0^\tau e^{-(\tau-s)H_B}\,\mathcal{E}(s)\,ds
\]
and using Gaussian bounds for $e^{-tH_B}$ on $B_{\mathrm{reg}}$ gives \eqref{eq:Rj-bound}.
\end{proof}

\subsection{Definition of the renormalized bilinear functional}\label{subsec:ren-def}

Let $\tau>0$ and let $\Phi,\Psi\in C_c^\infty(B)$ be smooth test functions, not assumed to avoid $D$.
(We identify $\Phi,\Psi$ with their restrictions to $B_{\mathrm{reg}}$ whenever needed.)

\paragraph{Radial cutoff at scale $\rho$.}
Fix a smooth function $\eta\in C^\infty([0,\infty))$ with $\eta\equiv 0$ on $[0,1/2]$ and $\eta\equiv 1$ on $[1,\infty)$.
For $\rho\in(0,r_0)$ define $\eta_\rho(r):=\eta(r/\rho)$.
Using wedge charts near each $p_j$, define a global cutoff $\eta_\rho\in C^\infty(B)$ by
\[
\eta_\rho(b):=
1-\sum_{j=1}^N (1-\eta_\rho(r_j(b)))\,\chi_j(b),
\]
where $r_j(\cdot)$ is the wedge radial coordinate near $p_j$ and $\chi_j$ are the cutoffs in the parametrix.
Since $D$ is finite, the sum is finite. For $\rho$ small enough $\eta_\rho$ satisfies
\[
\eta_\rho\equiv 0 \text{ on } \bigcup_j\{r_j\le \rho/2\},\qquad \eta_\rho\equiv 1 \text{ on } B\setminus \bigcup_j\{r_j\le \rho\}.
\]

Define
\begin{equation}\label{eq:Phi-split}
\Phi^{(\rho)}:=\eta_\rho\,\Phi,\qquad \Phi^{<\rho}:=\Phi-\Phi^{(\rho)},
\end{equation}
and similarly for $\Psi$.
Thus $\Phi^{(\rho)},\Psi^{(\rho)}$ vanish in a neighborhood of $D$, while $\Phi^{<\rho},\Psi^{<\rho}$
are supported in a union of wedge neighborhoods.

\begin{definition}[Renormalized bilinear functional]\label{def:Kren}
For $\tau>0$ and $\Phi,\Psi\in C_c^\infty(B)$ define
\begin{align}
K_B^{\mathrm{ren}}(\Phi,\Psi;\tau)
:=
\lim_{\rho\downarrow 0}\Bigg[
&\int_{B_{\mathrm{reg}}\times B_{\mathrm{reg}}}
K_B(b,b';\tau)\,\Phi^{(\rho)}(b)\,\Psi^{(\rho)}(b')\,d\mu_{g_B}(b)\,d\mu_{g_B}(b')
\label{eq:Kren-def}\\
&\quad
+\sum_{j=1}^N
\int_{U_j\times U_j}
\chi_j(b)\chi_j(b')\,K_j^{\mathrm{cone}}(b,b';\tau)\,
\Phi^{<\rho}(b)\,\Psi^{<\rho}(b')\,d\mu_{g^{\mathrm{cone}}_j}(b)\,d\mu_{g^{\mathrm{cone}}_j}(b')
\Bigg].\nonumber
\end{align}
\end{definition}

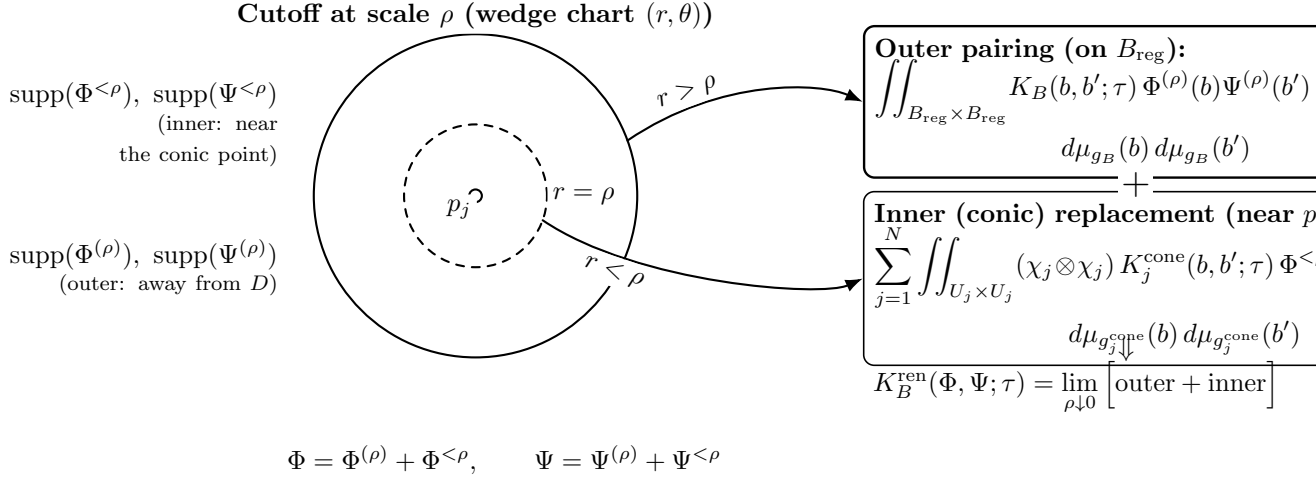
\begin{figure}[t]
\centering

\begin{tikzpicture}[
  >=Latex,
  font=\small,
  line cap=round,
  line join=round,
  box/.style={draw, rounded corners, align=left, text width=6.95cm, inner sep=4pt},
  boxouter/.style={box, line width=0.9pt},
  boxinner/.style={box, line width=0.6pt},
  lab/.style={fill=white, rounded corners=1pt, inner sep=2pt}
]

\def\Rbig{2.15cm}   
\def\Rrho{0.95cm}   
\coordinate (pj) at (0,0);

\draw[thick] (pj) circle (\Rbig);
\draw[thick, dashed] (pj) circle (\Rrho);

\draw[thick, fill=white] (pj) circle (2.4pt);
\node[lab, below left=-1pt] at (pj) {$p_j$};

\node[lab, anchor=west] at ($(pj)+(0:\Rrho)$) {$r=\rho$};

\node[lab, anchor=south] at ($(pj)+(0,\Rbig)$)
{\textbf{Cutoff at scale $\rho$ (wedge chart $(r,\theta)$)}};

\node[lab, anchor=east, align=right, text width=3.75cm] (innerlab)
at (-2.55cm,0.95cm)
{\(\supp(\Phi^{<\rho}),\ \supp(\Psi^{<\rho})\)\\[-0.6mm]\scriptsize(inner: near the conic point)};

\node[lab, anchor=east, align=right, text width=3.75cm] (outerlab)
at (-2.55cm,-0.95cm)
{\(\supp(\Phi^{(\rho)}),\ \supp(\Psi^{(\rho)})\)\\[-0.6mm]\scriptsize(outer: away from \(D\))};

\coordinate (outerstart) at ($(pj)+(20:\Rbig)$);
\coordinate (innerstart) at ($(pj)+(-20:\Rrho)$);

\node[boxouter, anchor=west] (outerbox) at (5.15cm,1.25cm)
{$\textbf{Outer pairing (on }B_{\mathrm{reg}}\textbf{):}$\\[-1mm]
\(\begin{aligned}
\iint_{B_{\mathrm{reg}}\times B_{\mathrm{reg}}}
&K_B(b,b';\tau)\,\Phi^{(\rho)}(b)\Psi^{(\rho)}(b')\\
&\qquad d\mu_{g_B}(b)\,d\mu_{g_B}(b')
\end{aligned}\)};

\node[boxinner, anchor=west] (innerbox) at (5.15cm,-1.10cm)
{$\textbf{Inner (conic) replacement (near }p_j\textbf{):}$\\[-1mm]
\(\begin{aligned}
\sum_{j=1}^N \iint_{U_j\times U_j}
&(\chi_j\!\otimes\!\chi_j)\,K^{\mathrm{cone}}_j(b,b';\tau)\,
\Phi^{<\rho}(b)\Psi^{<\rho}(b')\\
&\qquad d\mu_{g^{\mathrm{cone}}_j}(b)\,d\mu_{g^{\mathrm{cone}}_j}(b')
\end{aligned}\)};

\node[fill=white, inner sep=1pt] at ($(outerbox.south)!0.5!(innerbox.north)$) {\Large $+$};

\draw[->, thick]
  (outerstart)
  .. controls +(1.25cm,0.85cm) and +(-0.55cm,0.25cm) ..
  (outerbox.west)
  node[pos=0.23, above, sloped, lab] {$r>\rho$};

\draw[->, thick]
  (innerstart)
  .. controls +(1.25cm,-0.85cm) and +(-0.55cm,-0.25cm) ..
  (innerbox.west)
  node[pos=0.23, below, sloped, lab] {$r<\rho$};

\node at (8.65cm,-2.05cm) {$\Downarrow$};

\node[anchor=west] at (5.15cm,-2.55cm)
{$K_B^{\mathrm{ren}}(\Phi,\Psi;\tau)
=\displaystyle\lim_{\rho\downarrow0}\Big[\text{outer}+\text{inner}\Big]$};

\end{tikzpicture}

\vspace{0.35em}
{\small \(\Phi=\Phi^{(\rho)}+\Phi^{<\rho},\qquad \Psi=\Psi^{(\rho)}+\Psi^{<\rho}\)}

\caption{$\rho$--cutoff decomposition and the gluing mechanism in the definition of the renormalized functional $K_B^{\mathrm{ren}}$.
The test functions split into an outer part supported in $B_{\mathrm{reg}}$ and an inner part supported near the conic points.
The outer contribution uses the base heat kernel $K_B$, while the inner contribution is replaced by the model cone kernels $K_j^{\mathrm{cone}}$.
Taking $\rho\downarrow0$ yields a canonical cutoff-independent limit (Theorem~5.7).}
\label{fig:rho-cutoff-renorm}
\end{figure}

At this stage one must show that the limit exists and does not depend on \emph{any} auxiliary choice
(wedge charts, cutoffs $\chi_j$, or the cutoff profile $\eta$). This is the key intrinsic part of the construction.

\begin{theorem}[Existence and cutoff-independence]\label{thm:Kren-well-defined}
Assume Assumption~\ref{ass:conic-asymptotics} and the conic parametrix Proposition~\ref{prop:conic-parametrix}.
Then for every $\tau>0$ and every $\Phi,\Psi\in C_c^\infty(B)$, the limit in \eqref{eq:Kren-def} exists and is finite.

Moreover, $K_B^{\mathrm{ren}}(\Phi,\Psi;\tau)$ is independent of 
(i) the wedge charts $(r_j,\theta_j)$, 
(ii) the choice of cutoffs $\chi_j,\widetilde\chi_j$ used in the parametrix, and 
(iii) the cutoff profile $\eta$ used to build $\eta_\rho$.

Finally, if $\mathrm{supp}\,\Phi\cup \mathrm{supp}\,\Psi\subset B_{\mathrm{reg}}$, then for $\rho$ sufficiently small
$\Phi^{<\rho}=\Psi^{<\rho}=0$, hence
\[
K_B^{\mathrm{ren}}(\Phi,\Psi;\tau)=\int_{B_{\mathrm{reg}}\times B_{\mathrm{reg}}}
K_B(b,b';\tau)\,\Phi(b)\,\Psi(b')\,d\mu_{g_B}(b)\,d\mu_{g_B}(b').
\]
\end{theorem}

\begin{proof}
Fix $0<\rho'<\rho<r_0$.
Subtract the bracketed expressions in \eqref{eq:Kren-def} at $\rho$ and $\rho'$.
The difference is supported where at least one variable lies in an annulus $\{\rho'/2\le r_j\le \rho\}$ for some $j$.

On each $U_j\times U_j$ we insert the decomposition \eqref{eq:KB-parametrix}.
The conic pieces telescope thanks to the way $\Phi^{(\rho)}$ and $\Phi^{<\rho}$ partition $\Phi$ (and similarly for $\Psi$).
What remains are integrals of the remainders $R_j$ and the measure-comparison errors
$d\mu_{g_B}-d\mu_{g^{\mathrm{cone}}_j}$ against $\Phi,\Psi$ supported in $\{r_j<\rho\}$.
Using \eqref{eq:Rj-bound} and \eqref{eq:measure-compare}, and integrating against $d\mu_{g^{\mathrm{cone}}_j}^{\otimes 2}\sim (r\,dr\,d\theta)(r'\,dr'\,d\theta')$, we obtain
\[
\Big|\Delta_{\rho,\rho'}\Big|
\le
C(\tau)\,\|\Phi\|_{L^\infty}\|\Psi\|_{L^\infty}
\int_{0<r,r'<\rho} (r+r')^\beta \,(r\,dr)(r'\,dr')
\ \le\ C(\tau)\,\rho^{4+\beta}.
\]
Hence the net is Cauchy as $\rho\downarrow 0$, so the limit exists.

Independence of auxiliary choices follows from the same telescoping estimate:
changing charts/cutoffs changes the parametrix only by kernels satisfying the same bound as $R_j$,
hence the difference between two constructions is $O(\rho^{4+\beta})\to 0$.
If supports avoid $D$, then for $\rho$ below the distance to $D$ we have $\Phi^{<\rho}=\Psi^{<\rho}=0$,
and \eqref{eq:Kren-def} collapses to the bare pairing.
\end{proof}

\subsection{Identification of the total-space bilinear limit with $K_B^{\mathrm{ren}}$}\label{subsec:identify-total}

We now connect the base renormalized functional to the collapsed total space.
Let $(X_t,g(t))$ be the collapsing family as in Sections~\ref{sec:mosco}--\ref{sec:semigroup-kernel}, with elliptic torus
fibration $\Pi_t:X_t\to B_{\mathrm{reg}}$ on the regular locus, and let $K_t(\cdot,\cdot;\tau)$ be the heat kernel on $X_t$.
We consider the \emph{fiber-constant compressed bilinear pairing}
\begin{equation}\label{eq:total-bilinear}
\mathcal{K}_t(\Phi,\Psi;\tau)
:=\langle I_t\Phi, e^{-\tau H_t} I_t\Psi\rangle_{L^2_t},
\qquad \tau>0,
\end{equation}
where $I_t$ is the normalized lift from Section~\ref{sec:mosco} (on the regular locus) and $H_t$ is the Friedrichs generator of $E_t$.

\begin{assumption}[Interior collapse input away from $D$]\label{ass:interior-away-D}
For every $\rho>0$, define $B(\rho):=B\setminus\bigcup_j\{r_j\le \rho\}$, so $B(\rho)\Subset B_{\mathrm{reg}}$.
Assume that on $B(\rho)$ the hypotheses of Sections~\ref{sec:mosco}--\ref{sec:semigroup-kernel} hold uniformly:
semi-flat $C^1$ control, vertical spectral gap, and hence
\begin{equation}\label{eq:interior-conv-rho}
\langle I_t \phi, e^{-\tau H_t} I_t\psi\rangle_{L^2(\Pi_t^{-1}(B(\rho)))}
\longrightarrow
\int_{B(\rho)\times B(\rho)} K_B(b,b';\tau)\,\phi(b)\,\psi(b')\,d\mu_{g_B}(b)\,d\mu_{g_B}(b')
\end{equation}
for all $\phi,\psi\in C_c^\infty(B(\rho))$ and all $\tau>0$.
\end{assumption}

\begin{theorem}[Renormalized bilinear limit across $D$]\label{thm:total-to-Kren}
Assume Assumption~\ref{ass:interior-away-D} and the conic hypotheses of this section
(Assumption~\ref{ass:conic-asymptotics} and Proposition~\ref{prop:conic-parametrix}).
Then for every fixed $\tau>0$ and every $\Phi,\Psi\in C_c^\infty(B)$,
\begin{equation}\label{eq:total-limit-Kren}
\lim_{t\downarrow 0}\mathcal{K}_t(\Phi,\Psi;\tau)=K_B^{\mathrm{ren}}(\Phi,\Psi;\tau).
\end{equation}
Moreover, the convergence is uniform for $\tau$ in compact subsets of $(0,\infty)$.
\end{theorem}

\begin{proof}
Fix $\tau>0$ and split $\Phi,\Psi$ as in \eqref{eq:Phi-split} using a cutoff scale $\rho\in(0,r_0)$:
\[
\Phi=\Phi^{(\rho)}+\Phi^{<\rho},\qquad \Psi=\Psi^{(\rho)}+\Psi^{<\rho}.
\]
Expand $\mathcal{K}_t(\Phi,\Psi;\tau)$ into four terms.

\emph{Step 1: the outer--outer term.}
Since $\Phi^{(\rho)},\Psi^{(\rho)}$ are supported in $B(\rho/2)\Subset B_{\mathrm{reg}}$,
Assumption~\ref{ass:interior-away-D} gives
\[
\lim_{t\downarrow 0}\mathcal{K}_t(\Phi^{(\rho)},\Psi^{(\rho)};\tau)
=
\int_{B(\rho/2)\times B(\rho/2)}
K_B\,\Phi^{(\rho)}\otimes \Psi^{(\rho)}\,d\mu_{g_B}^{\otimes 2}.
\]

\emph{Step 2: mixed terms vanish as $\rho\downarrow 0$.}
Using selfadjointness/contractivity of $e^{-\tau H_t}$ and the $L^2$-isometry of $I_t$,
\[
\big|\mathcal{K}_t(\Phi^{(\rho)},\Psi^{<\rho};\tau)\big|
\le \|I_t\Phi^{(\rho)}\|_{L^2_t}\,\|I_t\Psi^{<\rho}\|_{L^2_t}
= \|\Phi^{(\rho)}\|_{L^2(B)}\,\|\Psi^{<\rho}\|_{L^2(B)}.
\]
Since $\Psi^{<\rho}$ is supported in $\bigcup_j\{r_j<\rho\}$ and $\Psi$ is smooth with compact support,
we have $\|\Psi^{<\rho}\|_{L^2(B)}\to 0$ as $\rho\downarrow 0$, hence the mixed term tends to $0$
uniformly in $t$. The same estimate applies to $\mathcal{K}_t(\Phi^{<\rho},\Psi^{(\rho)};\tau)$.

\emph{Step 3: the inner--inner term is captured by the cone model up to an error vanishing with $\rho$.}
Insert a partition of unity $\sum_j\chi_j$ on $\mathrm{supp}(\Phi^{<\rho})\cup\mathrm{supp}(\Psi^{<\rho})$ and work wedgewise.
On each $U_j$ compare the base kernel $K_B$ with the model cone kernel via Proposition~\ref{prop:conic-parametrix}.
The remainder estimate \eqref{eq:Rj-bound} and the measure comparison \eqref{eq:measure-compare} show that the error made by
replacing $K_B$ with $\chi_j\otimes\chi_j\,K_j^{\mathrm{cone}}$ in the inner--inner pairing is $O_\tau(\rho^{4+\beta})$.
Combining with interior convergence on truncated sets $B(\rho')$ and then letting $\rho'\downarrow 0$ yields
\[
\lim_{t\downarrow 0}\mathcal{K}_t(\Phi^{<\rho},\Psi^{<\rho};\tau)
=
\sum_{j=1}^N
\int_{U_j\times U_j}
\chi_j\otimes\chi_j\,K_j^{\mathrm{cone}}\,\Phi^{<\rho}\otimes\Psi^{<\rho}\,d\mu_{g^{\mathrm{cone}}_j}^{\otimes 2}
\;+\;O_\tau(\rho^{4+\beta}).
\]

\emph{Step 4: send $\rho\downarrow 0$.}
Combining Steps 1--3 and using the definition \eqref{eq:Kren-def} plus Theorem~\ref{thm:Kren-well-defined}
yields \eqref{eq:total-limit-Kren}.
\end{proof}

\subsection{Section summary}\label{subsec:sec5-summary}
We have established the following: 
(i) wedge charts and $C^1$ conic control of $g_B$ near $D$; 
(ii) a conic parametrix for $K_B$ with $(r+r')^\beta$-weighted remainder bounds; 
(iii) the renormalized bilinear functional $K_B^{\mathrm{ren}}(\Phi,\Psi;\tau)$ together with its cutoff-independence; and 
(iv) the identification of the collapsed total-space bilinear pairing limit with $K_B^{\mathrm{ren}}$ for each fixed $\tau>0$.

In the next section we will streamline the hypotheses needed in Assumption~\ref{ass:interior-away-D} in the geometric
semi-flat collapse setting, and record quantitative error bounds (in $t$) for the renormalized bilinear convergence.
\section{Quantitative estimates and error bookkeeping}\label{sec:quant-bookkeeping}

In this section we combine three mechanisms into a single quantitative comparison:
(i) the \emph{interior} (semi-flat) collapse estimate with rate 
$\epsilon_t \sim e^{-c/s^*(t)}$;
(ii) the \emph{edge} conic parametrix remainder of size $O(\rho^{4+\beta})$; and
(iii) the \emph{mixed} decoupling controlled by the $L^2$-smallness of the inner cutoffs 
(with off-diagonal bounds such as Davies--Gaffney; see 
\cite{DGRefs,Ouhabaz}).
We then state a final citable theorem encoding the iterated limit ``first $t\downarrow 0$, then $\rho\downarrow 0$''.

\subsection{Two-parameter truncation and the $\rho$--renormalized target}\label{subsec:rho-target}

Retain the wedge charts $(U_j,r_j,\theta_j)$, the cutoffs $\chi_j$, and the cutoff decomposition
$\Phi=\Phi^{(\rho)}+\Phi^{<\rho}$ from \S\ref{sec:conic-ren}. Fix $\tau>0$ and $\rho\in(0,r_0)$.

\paragraph{Total-space bilinear pairing.}
Let $(X_t,g(t))$ be the collapsing family and let $H_t$ be the Friedrichs generator of the Dirichlet form $E_t$
(with the boundary convention fixed in \S\ref{sec:mosco}--\S\ref{sec:semigroup-kernel}).
For $\Phi,\Psi\in C_c^\infty(B)$ define the compressed (fiber-constant) pairing
\begin{equation}\label{eq:Kt-bilinear-def}
\mathcal{K}_t(\Phi,\Psi;\tau)
:=\langle I_t\Phi,\,e^{-\tau H_t} I_t\Psi\rangle_{L^2(X_t)},
\end{equation}
where $I_t$ is the normalized lift.

\paragraph{$\rho$--renormalized base target.}
Define the \emph{$\rho$--renormalized} target functional
\begin{align}
K^{\mathrm{ren}}_{B,\rho}(\Phi,\Psi;\tau)
:=
&\int_{B_{\mathrm{reg}}\times B_{\mathrm{reg}}}
K_B(b,b';\tau)\,\Phi^{(\rho)}(b)\,\Psi^{(\rho)}(b')\,d\mu_{g_B}(b)\,d\mu_{g_B}(b')\label{eq:Kren-rho}\\
&\quad+
\sum_{j=1}^N\int_{U_j\times U_j}
\chi_j(b)\chi_j(b')\,K^{\mathrm{cone}}_j(b,b';\tau)\,
\Phi^{<\rho}(b)\,\Psi^{<\rho}(b')\,d\mu_{g^{\mathrm{cone}}_j}(b)\,d\mu_{g^{\mathrm{cone}}_j}(b').\nonumber
\end{align}
By Theorem~\ref{thm:Kren-well-defined}, $K^{\mathrm{ren}}_{B,\rho}(\Phi,\Psi;\tau)$ converges as $\rho\downarrow0$
to $K^{\mathrm{ren}}_B(\Phi,\Psi;\tau)$, independently of auxiliary choices.

\subsection{Quantitative interior estimate at fixed $\rho$}\label{subsec:interior-quant}

Let $\epsilon_t:=Ce^{-c/s^*(t)}$ be the semi-flat error gauge from \S\ref{sec:mosco}--\S\ref{sec:semigroup-kernel}.
Fix $\rho>0$ and consider the buffer region
\[
B(\rho/2):=B\setminus \bigcup_{j=1}^N\{r_j\le \rho/2\}\Subset B_{\mathrm{reg}}.
\]
Then $\Phi^{(\rho)},\Psi^{(\rho)}\in C_c^\infty(B(\rho/2))$, and all interior hypotheses hold uniformly on $B(\rho/2)$.

\begin{proposition}[Interior comparison with exponential rate]\label{prop:interior-rate-rho}
Fix $\rho\in(0,r_0)$ and a compact time window $\tau\in[\tau_1,\tau_2]\subset(0,\infty)$.
There exists $C=C(\rho,\tau_1,\tau_2)$ such that for all sufficiently small $t$,
\begin{equation}\label{eq:interior-rate}
\Big|\mathcal{K}_t(\Phi^{(\rho)},\Psi^{(\rho)};\tau)
-\!\!\int_{B_{\mathrm{reg}}\times B_{\mathrm{reg}}}\!\!
K_B(b,b';\tau)\,\Phi^{(\rho)}(b)\,\Psi^{(\rho)}(b')\,d\mu_{g_B}^{\otimes2}\Big|
\le C\,\epsilon_t\,\|\Phi\|_{L^2(B)}\|\Psi\|_{L^2(B)}.
\end{equation}
The estimate holds uniformly for $\tau\in[\tau_1,\tau_2]$.
\end{proposition}

\begin{proof}
On the fixed precompact set $B(\rho/2)$ we have the quantitative compressed semigroup bound
\[
\|P_t e^{-\tau H_t} I_t - e^{-\tau H_B}\|_{L^2\to L^2}\le C(\rho,\tau_1,\tau_2)\,\epsilon_t,
\qquad \tau\in[\tau_1,\tau_2],
\]
from \S\ref{sec:semigroup-kernel} (under the exponentially small semi-flat error hypothesis).
Pairing against $\Phi^{(\rho)},\Psi^{(\rho)}$ and using the $L^2$ control for $I_t$ yields \eqref{eq:interior-rate}.
\end{proof}

\subsection{Mixed-term control}\label{subsec:mixed-decoupling}

We bound the mixed pairings by contractivity and the $L^2$-smallness of the inner cutoffs.

\begin{lemma}[Mixed terms: contractive bound]\label{lem:mixed-contract}
For every $\tau>0$, all sufficiently small $t$, all $\rho\in(0,r_0)$, and all $\Phi,\Psi\in C_c^\infty(B)$,
\begin{align}
\Big|\mathcal{K}_t(\Phi^{(\rho)},\Psi^{<\rho};\tau)\Big|
&\le C\,\|\Phi^{(\rho)}\|_{L^2(B)}\,\|\Psi^{<\rho}\|_{L^2(B)},\label{eq:mixed-contract-1}\\
\Big|\mathcal{K}_t(\Phi^{<\rho},\Psi^{(\rho)};\tau)\Big|
&\le C\,\|\Phi^{<\rho}\|_{L^2(B)}\,\|\Psi^{(\rho)}\|_{L^2(B)}.\label{eq:mixed-contract-2}
\end{align}
Here $C>0$ is independent of $t,\rho$ (for $t$ small).
\end{lemma}

\begin{proof}
By $L^2$-contractivity of $e^{-\tau H_t}$ and Cauchy--Schwarz,
\[
\big|\mathcal{K}_t(\Phi^{(\rho)},\Psi^{<\rho};\tau)\big|
\le \|I_t\Phi^{(\rho)}\|_{L^2(X_t)}\,\|I_t\Psi^{<\rho}\|_{L^2(X_t)}.
\]
Using the $L^2$ control of $I_t$ (uniformly for $t$ small) yields \eqref{eq:mixed-contract-1}.
The second estimate is identical.
\end{proof}

Since in dimension two one has $\|\Phi^{<\rho}\|_{L^2(B)}=O(\rho)\|\Phi\|_{L^\infty(B)}$, the mixed terms vanish as $\rho\downarrow0$.

\begin{lemma}[$L^2$ size of the inner cutoff]\label{lem:L2-inner}
There exists $C>0$ such that for all sufficiently small $\rho$,
\begin{equation}\label{eq:L2-inner}
\|\Phi^{<\rho}\|_{L^2(B)}\le C\,\rho\,\|\Phi\|_{L^\infty(B)},\qquad
\|\Psi^{<\rho}\|_{L^2(B)}\le C\,\rho\,\|\Psi\|_{L^\infty(B)}.
\end{equation}
\end{lemma}

\begin{proof}
$\Phi^{<\rho}$ is supported in $\bigcup_j\{r_j<\rho\}$, whose $g_B$--area is $O(\rho^2)$ by the conic volume form
$d\mu_{g_B}=(1+O(r^\beta))\,\alpha_j r\,dr\,d\theta$ near each $p_j$.
Hence
\[
\|\Phi^{<\rho}\|_{L^2(B)}^2 \le \|\Phi\|_{L^\infty(B)}^2\,\mu_{g_B}\Big(\bigcup_j\{r_j<\rho\}\Big)
\le C\,\rho^2\,\|\Phi\|_{L^\infty(B)}^2,
\]
and similarly for $\Psi^{<\rho}$.
\end{proof}

\subsection{Edge parametrix remainder: the $O(\rho^{4+\beta})$ term}\label{subsec:edge-remainder}

We quantify the convergence of $K^{\mathrm{ren}}_{B,\rho}$ to $K^{\mathrm{ren}}_B$ and isolate the $\rho^{4+\beta}$ term coming
from the weighted remainder in the conic parametrix.

\begin{proposition}[Edge remainder bound]\label{prop:edge-rho}
Assume the conic asymptotics of \S\ref{sec:conic-ren} with exponent $\beta>0$.
Fix a compact time window $\tau\in[\tau_1,\tau_2]\subset(0,\infty)$.
Then there exists $C=C(\tau_1,\tau_2)$ such that for all $\rho\in(0,r_0)$ and all $\Phi,\Psi\in C_c^\infty(B)$,
\begin{equation}\label{eq:edge-rho-bound}
\big|K^{\mathrm{ren}}_{B,\rho}(\Phi,\Psi;\tau)-K^{\mathrm{ren}}_B(\Phi,\Psi;\tau)\big|
\le C\,\rho^{4+\beta}\,\|\Phi\|_{L^\infty(B)}\,\|\Psi\|_{L^\infty(B)},
\qquad \tau\in[\tau_1,\tau_2].
\end{equation}
\end{proposition}

\begin{proof}
This is the quantitative version of the Cauchy estimate in Theorem~\ref{thm:Kren-well-defined}.
Using the remainder bound \eqref{eq:Rj-bound} and the measure comparison near each $p_j$, the difference between two cutoff radii is controlled by
\[
C(\tau_1,\tau_2)\,\|\Phi\|_\infty\|\Psi\|_\infty
\int_{0<r,r'<\rho}(r+r')^\beta (r\,dr)(r'\,dr')
\ \lesssim\ \rho^{4+\beta}\|\Phi\|_\infty\|\Psi\|_\infty,
\]
uniformly for $\tau\in[\tau_1,\tau_2]$.
\end{proof}

\subsection{Final bookkeeping and the citable iterated-limit theorem}\label{subsec:final-theorem}

We can now combine the interior estimate, the mixed-term control, and the edge remainder into one statement.

\begin{theorem}[Quantitative renormalized comparison and iterated limit]\label{thm:main-iterated}
Assume:
\begin{enumerate}
\item (\emph{Semi-flat interior with exponential rate}) On every compact subset of $B_{\mathrm{reg}}$
the semi-flat collapse hypotheses of \S\ref{sec:mosco}--\S\ref{sec:semigroup-kernel} hold with error
$\epsilon_t\le Ce^{-c/s^*(t)}$ and yield the quantitative compressed semigroup estimate used in
Proposition~\ref{prop:interior-rate-rho}.
\item (\emph{Conic base model}) The base metric satisfies the $C^1$ conic asymptotics of
Assumption~\ref{ass:conic-asymptotics} with exponent $\beta>0$, and the conic parametrix
Proposition~\ref{prop:conic-parametrix} holds on a time interval containing $[\tau_1,\tau_2]$.
\end{enumerate}
Fix a compact time window $\tau\in[\tau_1,\tau_2]\subset(0,\infty)$. Then there exist constants
$C>0$ such that for all sufficiently small $t$ and all $\rho\in(0,r_0)$,
\begin{align}
\Big|\mathcal{K}_t(\Phi,\Psi;\tau)-K^{\mathrm{ren}}_{B,\rho}(\Phi,\Psi;\tau)\Big|
\le\;& C(\rho,\tau_1,\tau_2)\,\epsilon_t\,\|\Phi\|_{L^2(B)}\|\Psi\|_{L^2(B)} \label{eq:bookkeeping}\\
&\quad + C\,\Big(
\|\Phi^{(\rho)}\|_{L^2(B)}\|\Psi^{<\rho}\|_{L^2(B)}
+\|\Phi^{<\rho}\|_{L^2(B)}\|\Psi^{(\rho)}\|_{L^2(B)}\Big)\nonumber\\
&\quad + C\,\rho^{4+\beta}\,\|\Phi\|_{L^\infty(B)}\|\Psi\|_{L^\infty(B)}.\nonumber
\end{align}
Moreover, since $K^{\mathrm{ren}}_{B,\rho}(\Phi,\Psi;\tau)\to K^{\mathrm{ren}}_B(\Phi,\Psi;\tau)$ as $\rho\downarrow0$, we have the
\emph{iterated-limit identity}
\begin{equation}\label{eq:iterated-limit}
\lim_{\rho\downarrow0}\ \limsup_{t\downarrow0}\ 
\Big|\mathcal{K}_t(\Phi,\Psi;\tau)-K^{\mathrm{ren}}_B(\Phi,\Psi;\tau)\Big|=0,
\qquad \tau\in[\tau_1,\tau_2].
\end{equation}
In particular, $\mathcal{K}_t(\Phi,\Psi;\tau)\to K^{\mathrm{ren}}_B(\Phi,\Psi;\tau)$ for each fixed $\tau>0$.
\end{theorem}

\begin{proof}
Expand $\Phi=\Phi^{(\rho)}+\Phi^{<\rho}$ and $\Psi=\Psi^{(\rho)}+\Psi^{<\rho}$ in \eqref{eq:Kt-bilinear-def}.
This yields four terms. The outer--outer term is compared to the base $K_B$ pairing with rate $\epsilon_t$
by Proposition~\ref{prop:interior-rate-rho}.
The mixed terms are bounded by Lemma~\ref{lem:mixed-contract}.
For the inner--inner contribution, we add and subtract the conic model pairing; the difference is controlled by the conic
parametrix remainder estimate, yielding the $C\rho^{4+\beta}$ term (as in Proposition~\ref{prop:edge-rho}).
Collecting these bounds gives \eqref{eq:bookkeeping}.

To obtain \eqref{eq:iterated-limit}, fix $\rho>0$ and let $t\downarrow0$: the first term vanishes since $\epsilon_t\to0$.
Then let $\rho\downarrow0$: by Lemma~\ref{lem:L2-inner} the mixed term is $O(\rho)$, and the edge term is $O(\rho^{4+\beta})$.
Finally use Proposition~\ref{prop:edge-rho} to replace $K^{\mathrm{ren}}_{B,\rho}$ with $K^{\mathrm{ren}}_B$.
\end{proof}

\begin{remark}[A one-parameter corollary (optional)]\label{rem:rho-of-t}
If one wishes to turn the two-parameter bookkeeping into a single $t\downarrow0$ estimate, one may choose a function
$\rho=\rho(t)\downarrow0$ and combine \eqref{eq:bookkeeping}, Lemma~\ref{lem:L2-inner} and Proposition~\ref{prop:edge-rho}.
For instance, choosing $\rho(t):=\epsilon_t^{1/(4+\beta)}$ gives a concrete (non-optimal) bound
\[
\Big|\mathcal{K}_t(\Phi,\Psi;\tau)-K^{\mathrm{ren}}_B(\Phi,\Psi;\tau)\Big|
\le C\Big(\epsilon_t + \epsilon_t^{1/(4+\beta)}\Big),
\]
for fixed $\tau\in[\tau_1,\tau_2]$ and bounded $\Phi,\Psi$.
We keep the iterated-limit formulation \eqref{eq:iterated-limit} as the canonical statement.
\end{remark}

\begin{remark}[On small-time information]\label{rem:small-time-not-used}
Theorem~\ref{thm:main-iterated} is a fixed-$\tau$ statement and is designed to justify the renormalized limit across $D$.
Any additional analysis of the behavior of $K^{\mathrm{ren}}_B(\Phi,\Psi;\tau)$ as $\tau\downarrow0$ is purely base/model-driven
and is not used in the convergence arguments above.
\end{remark}
\section{Synthesis: renormalized limit for compressed heat flow}\label{sec:synthesis}

We now assemble the two complementary outputs established in
Theorem~\ref{thm:Kren-well-defined} (cutoff-independence and canonicity of $K_B^{\mathrm{ren}}$)
and Theorem~\ref{thm:main-iterated} (quantitative comparison and the fixed-$\tau$ collapsed limit)
into a single citable statement.

\begin{theorem}[Main theorem: renormalized compressed heat kernel limit]\label{thm:main-synthesis}
Assume:
\begin{enumerate}
\item (\emph{Semi-flat collapse and vertical gap on $B_{\mathrm{reg}}$})
on every precompact $B_0\Subset B_{\mathrm{reg}}$ the family $(X_t,g(t))$ admits semi-flat control
(with $\varepsilon_t\to0$; in addition one may assume an exponential gauge $\epsilon_t\le Ce^{-c/s^*(t)}$
when stating explicit rates), and the fiberwise first nonzero eigenvalue satisfies
$\lambda_1(F_{b,t})\gtrsim s^*(t)^{-2}$.
\item (\emph{Conic base structure at the discriminant})
the base metric $g_B$ satisfies the $C^1$ conic asymptotics near each $p_j\in D$ as in
Assumption~\ref{ass:conic-asymptotics}, so that the cutoff-independent functional
$K_B^{\mathrm{ren}}(\Phi,\Psi;\tau)$ of Definition~\ref{def:Kren} is well-defined.
\item (\emph{Off-diagonal control and bookkeeping framework})
the Davies--Gaffney decoupling and the quantitative bookkeeping hypotheses used in
Section~\ref{sec:quant-bookkeeping} hold on the relevant truncations away from $D$.
\end{enumerate}
Fix a boundary type among Dirichlet/Neumann (and Robin with nonnegative weight) as in
Sections~\ref{sec:semigroup-kernel}--\ref{sec:quant-bookkeeping}, and let $H_t$ denote the corresponding
Friedrichs generator on the truncated domains (with normalized lift $I_t$).

\smallskip
\noindent
\textbf{(i) Renormalized bilinear limit at fixed time.}
For every $\tau>0$ and every $\Phi,\Psi\in C_c^\infty(B)$,
\begin{equation}\label{eq:main-fixed-tau}
\mathcal K_t(\Phi,\Psi;\tau):=\langle I_t\Phi,\,e^{-\tau H_t}I_t\Psi\rangle_{L^2(X_t)}
\longrightarrow
K_B^{\mathrm{ren}}(\Phi,\Psi;\tau)
\qquad (t\downarrow0).
\end{equation}
If $\supp\Phi\cup\supp\Psi\subset B_{\mathrm{reg}}$, then the renormalization is trivial and
$K_B^{\mathrm{ren}}(\Phi,\Psi;\tau)$ coincides with the usual base pairing against $K_B(\cdot,\cdot;\tau)$.

\smallskip
\noindent
\textbf{(ii) Quantitative bookkeeping (two-parameter form).}
Fix a compact time window $\tau\in[\tau_1,\tau_2]\subset(0,\tau_0]$ as in Section~\ref{sec:quant-bookkeeping}.
Then for all sufficiently small $t$ and all $\rho\in(0,r_0)$ one has
\begin{equation}\label{eq:main-bookkeeping}
\big|\mathcal K_t(\Phi,\Psi;\tau)-K^{\mathrm{ren}}_{B,\rho}(\Phi,\Psi;\tau)\big|
\le
\mathrm{Err}_{\mathrm{int}}(t;\Phi,\Psi)
+
\mathrm{Err}_{\mathrm{mix}}(\rho,\tau;\Phi,\Psi)
+
\mathrm{Err}_{\mathrm{edge}}(\rho,\tau;\Phi,\Psi),
\end{equation}
with the same error structure as in Theorem~\ref{thm:main-iterated}. Consequently,
\[
\lim_{\rho\downarrow0}\limsup_{t\downarrow0}\big|\mathcal K_t(\Phi,\Psi;\tau)-K_B^{\mathrm{ren}}(\Phi,\Psi;\tau)\big|=0
\quad\text{uniformly for }\tau\in[\tau_1,\tau_2].
\]
\end{theorem}

\begin{proof}
This is exactly Theorem~\ref{thm:main-iterated}, together with the canonicity of $K_B^{\mathrm{ren}}$
from Theorem~\ref{thm:Kren-well-defined}.
\end{proof}

\begin{corollary}[Kernel pairing formulation]\label{cor:kernel-pairing}
On each precompact truncation where the heat kernel representation holds (Section~\ref{sec:semigroup-kernel}),
the convergence \eqref{eq:main-fixed-tau} is equivalent to the renormalized bilinear heat-kernel pairing limit
\[
\int_{X_{t,0}\times X_{t,0}}K_t(x,y;\tau)\,(I_t\Phi)(x)\,(I_t\Psi)(y)\,d\mu_{g(t)}(x)\,d\mu_{g(t)}(y)
\longrightarrow
K_B^{\mathrm{ren}}(\Phi,\Psi;\tau),
\qquad t\downarrow0,
\]
with the same error structure as in \eqref{eq:main-bookkeeping} after inserting the $\rho$--decomposition.
\end{corollary}
\section{Geometric input in semi-flat elliptic fibrations}\label{sec:geometric-verification}

This section records a concrete and verifiable set of geometric conditions implying the analytic hypotheses used in
Sections~\ref{sec:mosco}--\ref{sec:synthesis}. The goal is not to reprove semi-flat geometry, but to isolate the minimal
inputs needed to apply Theorem~\ref{thm:main-synthesis} in geometric collapse problems.

\subsection{Semi-flat approximation and the identification operators}\label{subsec:sf-identifications}

Let $\Pi_t:X_t\to B$ be a smooth elliptic fibration over $B_{\mathrm{reg}}:=B\setminus D$ with flat torus fibers
$F_{b,t}\cong \mathbb T^2$ and a family of Riemannian metrics $g(t)$ on $X_t$. Fix a precompact domain
$B_0\Subset B_{\mathrm{reg}}$ and set $X_{t,0}:=\Pi_t^{-1}(B_0)$.

Assume that on $X_{t,0}$ there exists a product reference metric
\[
g_t^{\Pi}:=g_B\oplus g_{F,t},
\]
as in the semi-flat collapse regimes studied in 
\cite{GrossWilson,TosattiAdiabatic,GrossTosattiZhang},
where $g_B$ is a fixed smooth metric on $B_0$ and $g_{F,t}$ restricts to a flat metric on each fiber, such that
\begin{equation}\label{eq:sf-C1-geo}
\|g(t)-g_t^{\Pi}\|_{C^1(X_{t,0},\,g_t^{\Pi})}\le \epsilon_t,\qquad
\epsilon_t\downarrow0
\quad
(\text{and optionally }\epsilon_t\le Ce^{-c/s^*(t)}).
\end{equation}

Let $A_t(b):=\mu_{F_{b,t}}(F_{b,t})$ be the fiber area and let $d\vartheta_{b,t}$ denote the Haar \emph{probability}
measure on $(F_{b,t},g_{F,t})$. Quantitative disintegration of volume along $\Pi_t$ yields measurable $\rho_t>0$
with fiberwise normalization $\int_{F_{b,t}}\rho_t\,d\vartheta_{b,t}=1$ such that for all $f\in L^1(X_{t,0})$,
\[
\int_{X_{t,0}} f\,d\mu_{g(t)}
=
\int_{B_0} A_t(b)\left(\int_{F_{b,t}} f\,\rho_t\,d\vartheta_{b,t}\right)d\mu_{g_B}(b).
\]
Moreover, the $C^1$ semi-flat control \eqref{eq:sf-C1-geo} implies the quantitative bounds
\[
\|\rho_t-1\|_{L^\infty(X_{t,0})}
+\|\nabla_{g_B}\rho_t\|_{L^\infty(X_{t,0})}
+\|\nabla_{g_B}\log A_t\|_{L^\infty(B_0)}
\ \le\ C\,\epsilon_t,
\]
after choosing the standard horizontal/vertical splitting associated with $g_t^{\Pi}$.

The normalized lift/average operators are then
\[
(I_t v)(x)=A_t(\Pi_tx)^{-1/2}v(\Pi_tx),\qquad
(P_t u)(b)=A_t(b)^{1/2}\int_{F_{b,t}}u\,\rho_t\,d\vartheta_{b,t},
\]
so that $P_tI_t=\mathrm{Id}$ and $I_t=P_t^*$, with two-sided $L^2$ near-isometry and $H^1$ compatibility on $B_0$.
These are exactly the structural inputs used to formulate Mosco convergence and compressed semigroup limits.

\subsection{Vertical spectral gap from fiber scale}\label{subsec:vertical-gap}

The collapse scale $s^*(t)\downarrow0$ is encoded in a uniform upper bound on the fiber diameter:
\begin{equation}\label{eq:fiber-diam-geo}
\mathrm{diam}(F_{b,t},g_{F,t})\le s^*(t)\qquad \text{for all }b\in B_0.
\end{equation}

\begin{lemma}[Flat torus gap from diameter]\label{lem:torus-gap-geo}
There exists a universal constant $c_0>0$ such that for every flat $2$--torus $(\mathbb T^2,h)$,
\[
\lambda_1(\mathbb T^2,h)\ge c_0\,\mathrm{diam}(\mathbb T^2,h)^{-2},
\]
where $\lambda_1$ denotes the first nonzero eigenvalue of $-\Delta_h$.
Consequently, under \eqref{eq:fiber-diam-geo},
\begin{equation}\label{eq:gap-from-scale-geo}
\lambda_1(F_{b,t},g_{F,t})\ge c_0\,s^*(t)^{-2}\qquad \text{uniformly for }b\in B_0.
\end{equation}
\end{lemma}

\begin{proof}
This follows, for instance, from the Zhong--Yang lower bound $\lambda_1\ge \pi^2/\mathrm{diam}^2$
on compact manifolds with $\mathrm{Ric}\ge0$ \cite{ZhongYang}, applied to flat tori (for which $\mathrm{Ric}=0$).
\end{proof}

Lemma~\ref{lem:torus-gap-geo} yields the fiberwise Poincar\'e inequality and the suppression of fiber-nonconstant modes at
energies $\ll s^*(t)^{-2}$, which is the starting point of the analytic chain in Sections~\ref{sec:mosco}--\ref{sec:semigroup-kernel}.

\subsection{Conic base input at the discriminant}\label{subsec:conic-input}

Near each discriminant point $p_j\in D$, fix a wedge chart $U_j$ with polar coordinate $r=r_j$ and assume that
$g_B$ is $C^1$--close to the flat cone metric
\[
g_j^{\mathrm{cone}}:=dr^2+\alpha_j^2 r^2\,d\theta^2
\]
in the sense of Assumption~\ref{ass:conic-asymptotics}, with some exponent $\beta>0$ controlling the coefficient defect.
This is precisely the hypothesis needed to construct the conic parametrix and the cutoff-independent functional
$K_B^{\mathrm{ren}}(\cdot,\cdot;\tau)$ (Section~\ref{sec:conic-ren}).

\subsection{Application template}\label{subsec:application-template}

In summary, to apply Theorem~\ref{thm:main-synthesis} it suffices to verify on each $B_0\Subset B_{\mathrm{reg}}$:
\begin{enumerate}
\item semi-flat $C^1$ control \eqref{eq:sf-C1-geo} for $g(t)$ relative to $g_B\oplus g_{F,t}$ (with $\epsilon_t\to0$);
\item the induced disintegration bounds on $(A_t,\rho_t)$ including $\|\nabla_{g_B}\rho_t\|_\infty=O(\epsilon_t)$;
\item fiber scale control \eqref{eq:fiber-diam-geo}, hence the vertical gap \eqref{eq:gap-from-scale-geo};
\item a $C^1$ conic model for $g_B$ near each $p_j\in D$ with cone parameter $\alpha_j$.
\end{enumerate}
Once these are in place, the analytic mechanism (vertical suppression $\Rightarrow$ Mosco $\Rightarrow$ compressed
semigroups $\Rightarrow$ interior bilinear kernel limit $\Rightarrow$ conic renormalization) applies verbatim, with the
quantitative bookkeeping recorded in Section~\ref{sec:quant-bookkeeping}.

\end{document}